\pdfoutput=1
\newif\ifpersonal
\personaltrue 
\documentclass[12pt,a4paper,dvipsnames,x11names]{amsart} 
\usepackage{amsmath,amsthm,amssymb,mathrsfs,mathtools,bm,eucal,tensor} 
\usepackage{stmaryrd}
\usepackage[all,cmtip]{xy}
\usepackage{dsfont}
\usepackage{microtype,lmodern} 
\usepackage[utf8]{inputenc} 
\usepackage[T1]{fontenc} 
\usepackage{enumerate,comment,braket,xspace,tikz-cd,csquotes} 
\usepackage[centering,vscale=0.75,hscale=0.75]{geometry}
\usepackage{xcolor}
\usepackage[hidelinks,urlcolor=black,colorlinks=true,linkcolor=Red4,citecolor=RoyalBlue]{hyperref}
\usepackage[capitalize]{cleveref}
\usepackage{mathpazo}
\usepackage{euler}

\usepackage{aliascnt} 

\newaliascnt{thm}{subsection}
\newtheorem{thm}[thm]{Theorem}
\aliascntresetthe{thm}
\crefname{thm}{Theorem}{Theorems}

\newaliascnt{theorem}{subsection}
\newtheorem{theorem}[theorem]{Theorem}
\aliascntresetthe{theorem}
\crefname{theorem}{Theorem}{Theorems}

\newaliascnt{prop}{subsection}
\newtheorem{prop}[prop]{Proposition}
\aliascntresetthe{prop}
\crefname{prop}{Proposition}{Propositions}

\newaliascnt{proposition}{subsection}

\aliascntresetthe{proposition}
\crefname{proposition}{Proposition}{Propositions}

\newaliascnt{cor}{subsection}
\newtheorem{cor}[cor]{Corollary}
\aliascntresetthe{cor}
\crefname{cor}{Corollary}{Corollaries}

\newaliascnt{lem}{subsection}
\newtheorem{lem}[lem]{Lemma}
\aliascntresetthe{lem}
\crefname{lem}{Lemma}{Lemmas}

\newaliascnt{lemma}{subsection}

\aliascntresetthe{lemma}
\crefname{lemma}{Lemma}{Lemmas}

\newaliascnt{question}{subsection}
\newtheorem{question}[question]{Question}
\aliascntresetthe{question}
\crefname{question}{Question}{Questions}

\newaliascnt{conjecture}{subsection}

\aliascntresetthe{conjecture}
\crefname{conjecture}{Conjecture}{Conjectures}

\newaliascnt{eg}{subsection}

\aliascntresetthe{eg}
\crefname{eg}{Example}{Examples}

\theoremstyle{definition}

\newaliascnt{defin}{subsection}
\newtheorem{defin}[defin]{Definition}
\aliascntresetthe{defin}
\crefname{defin}{Definition}{Definitions}

\newaliascnt{definition}{subsection}
\newtheorem{definition}[definition]{Definition}
\aliascntresetthe{definition}
\crefname{definition}{Definition}{Definitions}

\newaliascnt{assumption}{subsection}
\newtheorem{assumption}[assumption]{Assumption}
\aliascntresetthe{assumption}
\crefname{assumption}{Assumption}{Assumptions}

\newaliascnt{recollection}{subsection}
\newtheorem{recollection}[recollection]{Recollection}
\aliascntresetthe{recollection}
\crefname{recollection}{Recollection}{Recollections}

\newaliascnt{observation}{subsection}

\aliascntresetthe{observation}
\crefname{observation}{Observation}{Observations}

\newaliascnt{situation}{subsection}

\aliascntresetthe{situation}
\crefname{situation}{Situation}{Situations}

\newaliascnt{remark}{subsection}

\aliascntresetthe{remark}
\crefname{remark}{Remark}{Remarks}

\newaliascnt{rem}{subsection}
\newtheorem{rem}[rem]{Remark}
\aliascntresetthe{rem}
\crefname{rem}{Remark}{Remarks}

\newaliascnt{remarks}{subsection}

\aliascntresetthe{remarks}
\crefname{remarks}{Remarks}{Remarks}

\newaliascnt{example}{subsection}
\newtheorem{example}[example]{Example}
\aliascntresetthe{example}
\crefname{example}{Example}{Examples}

\newaliascnt{construction}{subsection}
\newtheorem{construction}[construction]{Construction}
\aliascntresetthe{construction}
\crefname{construction}{Construction}{Constructions}

\newaliascnt{defth}{subsection}

\aliascntresetthe{defth}
\crefname{defth}{Definition-Theorem}{Definition-Theorems}

\numberwithin{equation}{subsection}

\newaliascnt{notation}{subsection}
\newtheorem{notation}[notation]{Notation}
\aliascntresetthe{notation}
\crefname{notation}{Notation}{Notations}

\newaliascnt{linear}{subsection}

\aliascntresetthe{linear}
\crefname{linear}{Linear overview}{Linear overviews}

\usepackage{etoolbox}
\patchcmd{\section}{\scshape}{\bfseries}{}{}
\makeatletter
\renewcommand{\@secnumfont}{\bfseries}
\makeatother

\newcommand{\sep}{\mathrm{sep}}

\newcommand{\cl}{\mathrm{cl}}
\newcommand{\fm}{\mathfrak m}

\newcommand{\Hom}{\mathrm{Hom}}
\newcommand{\id}{\mathrm{id}}

\newcommand{\dt}{\mathrm{dimtot}}

\newcommand{\sw}{\mathrm{sw}}

\newcommand{\sO}{\mathscr{O}}

\newcommand{\cF}{\mathcal F}
\newcommand{\fc}{\mu}
\newcommand{\fd}{P}

\newcommand{\bZ}{\mathbb Z}

\newcommand{\bA}{\mathbb A}
\newcommand{\bP}{\mathbb P}
\newcommand{\bQ}{\mathbb Q}
\newcommand{\bT}{\mathbb T}
\newcommand{\bG}{\mathbb G}

\newcommand{\bF}{\mathbb F}
\newcommand{\bN}{\mathbb N}

\newcommand{\bC}{\mathbb C}

\newcommand{\cE}{\mathcal E}

\newcommand{\cM}{\mathcal M}
\newcommand{\cK}{\mathcal K}
\newcommand{\cP}{\mathcal P}

\newcommand{\cO}{\mathcal O}
\newcommand{\cC}{\mathcal C}
\newcommand{\cH}{\mathcal H}
\newcommand{\cL}{\mathcal L}
\newcommand{\cI}{\mathcal I}

\DeclareMathOperator{\et}{\text{ét}}
\DeclareMathOperator{\Fl}{Fl}
\DeclareMathOperator{\Ens}{Set}
\DeclareMathOperator{\Sch}{Sch}
\DeclareMathOperator{\Sym}{Sym}
\DeclareMathOperator{\Tr}{Tr}

\DeclareMathOperator{\rk}{rk}

\DeclareMathOperator{\topo}{top}

\DeclareMathOperator{\sbar}{\overline{s}}
\DeclareMathOperator{\ubar}{\overline{u}}
\DeclareMathOperator{\xbar}{\overline{x}}

\DeclareMathOperator{\Cons}{Cons}
\DeclareMathOperator{\lc}{lc}
\DeclareMathOperator{\Perv}{Perv}

\newcommand{\Qlbar}{\overline{\bQ}_\ell}  

\DeclareMathOperator{\Coh}{Coh}
\DeclareMathOperator{\totdeg}{totdeg}

\DeclareMathOperator{\length}{length}
\DeclareMathOperator{\tors}{tors}

\DeclareMathOperator{\GL}{GL}

\DeclareMathOperator{\Spec}{Spec}

\DeclareMathOperator{\dimtot}{dimtot}

\DeclareMathOperator{\Rk}{Rk}
\DeclareMathOperator{\Supp}{Supp}
\DeclareMathOperator{\Weil}{Weil}

\DeclareMathOperator{\LC}{Loc}

\DeclareMathOperator{\Loc}{Loc}
\DeclareMathOperator{\Gr}{Gr}
\DeclareMathOperator{\ft}{tf}

\DeclareMathOperator{\op}{op}
\DeclareMathOperator{\Fun}{Fun}
\DeclareMathOperator{\cHom}{\cH om}

\usepackage[english]{babel}

\title[]{Characteristic cycle and wild Lefschetz theorems}

\begin{document}

\author[H.Hu]{Haoyu Hu} 
\address{School of Mathematics, Nanjing University, Hankou Road 22, Nanjing, China}
\email{huhaoyu@nju.edu.cn, huhaoyu1987@gmail.com}

\author[J.-B. Teyssier]{Jean-Baptiste Teyssier}
\address{Institut de Math\'ematiques de Jussieu, 
Institut universitaire de France (IUF)}
\email{jean-baptiste.teyssier@imj-prg.fr}

\setcounter{tocdepth}{1}
\begin{abstract}
By relying on a new approach to  Lefschetz type questions based on Beilinson's singular support and Saito's characteristic cycle, we prove an instance of the wild Lefschetz theorem envisioned by Deligne.
Our main tool are new finiteness results for the characteristic cycles of perverse sheaves.
\end{abstract}

\maketitle

\tableofcontents

\section{Introduction}
This paper is a contribution to Lefschetz type theorems in positive characteristic by means of Beilinson's singular support and Saito's characteristic cycle. \medskip

Let $X \subset \bP_{\bC}$ be a smooth complex connected projective variety of dimension $\geq 2$.
By a fundamental result of Lefschetz, 
for every sufficiently generic hyperplane $H$, the induced map at the level of the fundamental groups
$$
\pi_1^{\topo}(X\cap H)\to \pi^{\topo}_1(X)
$$ 
is  surjective.
When $X$ is quasi-projective, Lefschetz's theorem still holds by works of Goresky-MacPherson \cite{GM} and Hamm-Lê \cite{HL}.
A natural question is to know if a similar statement holds for the étale fundamental group over some algebraically closed field $k$ of characteristic $p>0$.
See \cite{ELefschetz} for a survey of this question.
In the projective case, a positive answer was given by Grothendieck \cite{SGA2}.
In the quasi-projective case however, the answer is already negative for $U=\bA_k^2$ as shown in \cite[Lemma 5.4]{ELefschetz}.
When $U$ is the complement of a strict normal crossing divisor $D$ in a smooth projective variety $X\subset \bP_k$, the Lefschetz theorem can nonetheless be saved  if instead of considering the full étale fundamental group $\pi_1^{\et}(U)$ one considers its quotient $\pi_1^{t}(U)$ classifying étale covers of $U$ with \textit{tame} ramification along $D$ \cite{DriDel,EKin}.
Beyond the tame case, wild ramification needs to be bounded.
A natural way to do this is by means of an effective Cartier divisor $R$ supported on $D$.
\begin{defin}\label{def_provisoire}
We say  that $\cL\in \Loc(U,\Qlbar)$ has \textit{log conductors bounded by $R$} if for every morphism $f : C\to X$ over $k$ where $C$ is a smooth curve over k not mapped inside $D$ and for every $x\in f^{-1}(D)$, the higher logarithmic ramification slope of $\cL|_{C\times_X U}$ at $x$ as in   \cite{CL} is smaller than the multiplicity of $f^*R$ at $x$.
\end{defin}

     Note that \cref{def_provisoire} makes sense even for $X$ singular and $\Qlbar$ replaced by a finite local ring  $\Lambda$ of residue characteristic $\ell \neq p$.
     The following remark ensures that when $X$ is smooth, checking  log conductors boundedness by $R$ is easy.
\begin{rem}[{\cite[Proposition 5.8]{HuTeyssierCohBoundedness}}]
If $X$ is smooth over $k$, the sheaf $\cL\in \Loc(U,\Qlbar)$ has log conductors bounded by $R$  if for every irreducible component $Z$ of $D$, the logarithmic conductor of $\cL$ at the generic point of $Z$ is smaller than the multiplicity of $Z$ in $R-D$.
\end{rem}

\medskip

      As a replacement for surjectivity, 
Deligne asked \cite[Introduction]{ES}  the following

\begin{question}\label{q_deligne}
Let $X\subset \bP_k$ be a normal projective variety and let $U$ be the complement of an effective Cartier divisor $D$ of $X$.
Given $r\geq 0$ and an effective Cartier divisor $R$ supported on  $D$, can one  find a hypersurface $H\subset \bP_k$ such that for every $\cL\in \Loc(U,\Qlbar)$ with rank $\leq r$ and log conductors bounded by $R$, the sheaves $\cL$ and $\cL|_{U\cap H}$ have the same monodromy groups?
\end{question}
\cref{q_deligne} was answered positively in the rank 1 case in \cite{KS} when the compactification $X$  is smooth and $D$ is a strict normal crossing divisor and in \cite{ES} in the rank 1 case when $X$ is normal and $D$ is an effective Cartier divisor.\medskip

The following theorem offers the first instance beyond the abelian case of a theorem of Lefschetz type for wildly ramified local systems.

\begin{thm}\label{thm_5}
Let $X$ be a projective scheme of pure dimension $n\geq 2$ over $k$ algebraically closed.
Let $j : U\hookrightarrow X$ be an open immersion with $U$ smooth connected over $k$ such that $D:=X-U$ is the support of an effective Cartier divisor.
Then, there is a closed immersion $i : X\hookrightarrow \bP_k$ such that for  every finite local ring  $\Lambda$ of residue characteristic $\ell \neq p$, there is a function $N : \bN^2 \to \bN$ such that for every $r\geq 0$, every effective Cartier divisor $R$ supported on $D$, there is a dense open subset $V\subset (\bP_k^{\vee})^{N(\deg R,r)}$ satisfying the following: for every $(H_1,\dots, H_{N(\deg R,r)})\in V$ and every $\cL\in \Loc_{tf}(U,\Lambda)$ of rank $\leq r$ with log conductors bounded by $R$, there is $1\leq  a \leq N(\deg R,r)$ such that $\cL$ and $\cL|_{U\cap H_a}$ have the same monodromy group.
\end{thm}

Ignoring the quantitative aspect of \cref{thm_5},  one gets the following
\begin{thm}\label{thm_5bis}
In the setting of \cref{thm_5},  there is a closed  immersion $i : X\hookrightarrow \bP_k$  such that for every prime $\ell \neq p$,  every $r\geq 0$ and every effective Cartier divisor $R$ supported on $D$,  there exist hyperplanes $(H_1,\dots, H_{N})$ with $N$ depending on $\ell,  r$ and $R$ such that for every  $\cL\in  \Loc(U,\bF_{\ell})$ of rank $\leq r$ with log conductors bounded by $R$,  one of the $H_i$ preserves the monodromy group of $\cL$.
\end{thm}

Let us explain the differences between \cref{thm_5} and  Deligne's original question.
Firstly, \cref{thm_5} applies to finite coefficients   while Deligne asks for $\Qlbar$-coefficients.
To explain what makes the finite coefficient case already meaningful, observe that the objects of $\Loc_{tf}(U,\Lambda)$ with log conductors bounded by $R$ are the representations of a quotient $\pi_1(U,R)$ of $\pi_1(U)$ classifying finite étale covers with ramification bounded by $R$ (see \cite{Hiranouchi}).
Artin-Schreier's theory implies that the pro-p abelian quotient of $\pi_1(U,R)$ is not topologically finitely generated in general.
Thus, $\pi_1(U,R)$ is not topologically finitely generated in general.
Hence, even in the finite coefficient case, defining an object of $\Loc_{tf}(U,\Lambda)$ requires to specify an \textit{infinite} number of data.
What \cref{thm_5} says is that these data are captured by  a \textit{finite} number of hyperplane sections.
Secondly, Deligne's question is about the existence of a single hypersurface of possibly high degree depending on the divisor $R$, whereas \cref{thm_5} sticks to hyperplane sections for an embedding \textit{independent} of $R$. 
Also,  we allow some flexibility in the choice of these hyperplanes.
The price to pay for using only hyperplanes is that several are needed.

\begin{rem}
If one is interested by cutting down the dimension further with higher codimension projective subspaces,  one can construct a moduli of multi-flags containing a dense open subset of points realizing the Lefschetz theorem.
This moduli provides some flexibility useful in practice.  
Using it, one can indeed show that the projective subspaces realizing the Lefschetz theorem can be taken in some prescribed dense open subset of the Grassmannian.
See \cref{Wild_Lefschetz} for the full statement.
\end{rem}

The proofs of \cref{thm_5} rely on a new approach to questions of Lefschetz type based on Beilinson's singular support and Saito's characteristic cycle.
Namely given a  quasi-projective variety $U\subset \bP_k$, a hyperplane $H$ and $\cL\in \Loc_{tf}(U,\Lambda)$, we give a criterion for $\cL$ and $\cL|_{U\cap H}$  to have  the same monodromy group in terms of the transversality of $ H \hookrightarrow \bP_k$ with respect to the singular support of some auxiliary sheaf constructed from $\cL$ (see \cref{sheaf_criterion_same_monodromy} and \cref{transversality_and_Hom}).
The whole point of \cref{thm_5} is then to produce enough hyperplanes satisfying this transversality criterion, or equivalently to show that the singular supports under scrutiny can be controlled in some sense.\medskip

To explain how to do this, let us recall Beilinson's singular support construction in algebraic geometry. 
For a transcendental construction predating Beilinson's construction see \cite{Kashiwara_Schapira}.
For a smooth scheme of finite type $X$ over a perfect field $k$ of characteristic $p>0$ and for a finite field $\Lambda$ of characteristic $\ell\neq p$,  Beilinson associates  to every object $\cK\in D_{c}^{b}(X,\Lambda)$  a closed conical subset $SS(\cK)$ of the cotangent bundle $\bT^*X$ of $X$, called the \textit{singular support of $\cK$} and recording the hyperplanes through which the étale local sections of $\cK$ don't extend \cite{bei}.
In particular $SS(\cK)$ measures how far $\cK$ is from being locally constant.
The propagation defect of sections through a hyperplane $H$ is reflected by the non triviality of some monodromy action on the complex of vanishing cycles of $\cK$.
The \textit{characteristic cycle $CC(\cK)$ of $\cK$} is a cycle supported  on $SS(\cK)$ measuring the wild ramification of this monodromy action through the so-called Milnor formula \cite[Theorem 5.9]{cc}.\\ \indent
Thus, for every $\cK$ as above,  basic  numerical data can be attached : the number of irreducible components of $SS(\cK)$ and the multiplicities of $CC(\cK)$.
To make sense of a third basic set of data playing a crucial role in this paper, let us recall following \cite{bei} that for a sufficiently nice closed immersion $i : X\hookrightarrow \bP_k$ in some projective space  (see \cref{relative_beilinson_assumption}) and for every $\cK\in D_{c}^{b}(\bP_k,\Lambda)$ supported on $X$ with 
$$
CC(\cK)= \sum_a  m_a\cdot [C_a]
$$ 
where the $C_a \subset \bT^*\bP_k$ are distinct closed  irreducible conical subsets, the images $D_a$ of the projectivisation $\bP(C_a)\subset  \bP(\bT^*\bP_k)$ by the map $p^{\vee} : \bP(\bT^*\bP_k)\to \bP_k^{\vee}$ sending $(x,H)$ to $H$ is a divisor of $\bP_k^{\vee}$  and the induced map $\bP(C_a)\to D_a$ is generically radicial.
In a nutshell, this says that sufficiently nice embeddings have enough hyperplanes to distinguish the components of $CC$ for sheaves supported on $X$.
Hence, one may further consider the degrees $d_a$ of the maps $\bP(C_a)\to D_a$ and the degrees of the divisors $D_a$.
Let us package these numerical data into the \textit{total degree of $CC(\cK)$ with respect to $i : X\hookrightarrow \bP_k$} defined by
$$
\totdeg_i(CC(\cK)):=(-1)^{n-m} \sum_{a}  d_{a}  \cdot m_{a}  \cdot \deg D_{a}  
$$
where $n=\dim X$ and $m =\dim \bP_k$.
From the perspective of \cref{thm_5}, the degrees $\deg D_a$ are the numerical data we need to bound.
As explained above, proving \cref{thm_5} indeed amounts to construct enough hyperplanes $H$ transversal to some suitable singular supports $SS(\cK)$.
On the other hand transversality of $H$ with respect to $SS(\cK)$ is equivalent to ask for $H$ to avoid all the $D_a$.
By some basic algebraic geometry lemma (see \cref{points_avoinding_hypersurface_absolute}), this can be achieved if we take enough hyperplanes at the cost of bounding the $\deg D_a$.
Since for a perverse sheaf $\cK$, the multiplicities of $CC(\cK)$ are positive integers \cite[Proposition 5.14]{cc}, it is enough to bound the total degree $\totdeg_i(CC(\cK))$.
To carry out this program, we provide the following  cohomological interpretation of the total degree : 

\begin{thm}[\cref{chiXL}]\label{thm_7}
Let $X$ be a smooth projective scheme of pure dimension over an algebraically closed field $k$.
Let $i :X\hookrightarrow \bP_k$ be a closed immersion satisfying \cref{relative_beilinson_assumption}.
Let $\Lambda$ be a finite field of characteristic $\ell \neq p$.
Then, for every $\cK\in D_{c}^{b}(X,\Lambda)$, for every sufficiently generic hyperplane $H\subset \bP_k$ and every  sufficiently generic pencil $L\subset \bP_k^{\vee}$, we have
$$
\totdeg_i(CC(i_*\cK)) =\chi(X_L, \cK|_{X_L})  -2 \cdot \chi(X\cap H, \cK|_{X\cap H} ) \ .
$$
where $X_L$ is the total space of the pencil $L$.
\end{thm}

Hence, \cref{thm_7} \textit{translates the problem of bounding the total degree into the problem of bounding some Euler-Poincaré characteristics}.
That this latter boundedness holds under rank and ramification boundedness conditions follows from the main result of \cite{HuTeyssierMinkowski}.\medskip

To take a full advantage of it, one needs to bound the wild ramification not only for the extension by 0 of locally constant constructible sheaves but for arbitrary constructible complexes.
To this end, effective Cartier divisors are insufficient as some wild ramification may hide in codimension $>1$.
To solve this problem, we introduced in \cite{HuTeyssierCohBoundedness} some ramification boundedness using coherent sheaves instead of effective Cartier divisors. 
If $\bQ[\Coh(X)]$ is the free $\bQ$-vector space on the set of isomorphism classes of coherent sheaves on a scheme of finite type $X$ over a field $k$, we set the following 

\begin{defin}\label{defin_E_bound}
For $\cE\in \bQ[\Coh(X)]$, we say that $\cK\in D_{c}^{b}(X,\Lambda)$  has \textit{log conductors bounded by $\cE$} if for every $i\in \bZ$, every morphism $f : C\to X$ where $C$ is a smooth curve over $k$ and every $x\in C$, the logarithmic conductor (see \cref{local_fields}) of $\cH^i\cK|_C$ at $x$ is smaller than the length of the torsion part of $(f^*\cE)_x$ viewed as a module over $\cO_{C,x}$.
\end{defin}

By combining \cite[Corollary 7.30]{HuTeyssierMinkowski} with \cref{thm_7}, we get the following

\begin{thm}[\cref{Bounddmdeg_normal}]\label{thm_8} 
In the setting of \cref{thm_7}, let $\Sigma$ be a stratification of $X$ and let $a\leq b$ be integers.
Then, there is a function $\fc : \bQ[\Coh(X)]\to \bQ$ and  $\fd\in \mathds{N}[x]$ of degree $\dim X$ such that for every finite field $\Lambda$ of  characteristic $\ell \neq p$, every $\cE\in \bQ[\Coh(X)]$ and every $\cK\in D_{\Sigma}^{[a,b]}(X,\cE,\Lambda)$,  we have
$$
|\totdeg_i(CC(i_*\cK)) | \leq \fd(\fc(\cE)) \cdot \Rk_{\Lambda} \cK \ .
$$
In particular, if $\cK$ is perverse, the number of irreducible components of $SS(\cK)$, the multiplicities of $CC(\cK)$ and the degrees of the $D_a$ are smaller than $\fd(\fc(\cE)) \cdot \Rk_{\Lambda} \cK$.
\end{thm}

Hence, if we fix a stratification and bound the rank and the wild ramification,  there is a universal bound  for the degrees of the divisors $p^{\vee}(SS(i_*\cK))\subset \bP_k^{\vee}$,  which we saw is what is needed to prove \cref{thm_5}.

\subsection*{Acknowledgement}
We thank A. Abbes,  M. D'Addezio,  H. Esnault, M. Kerz, D. Litt and T. Saito for their interest and their comments.
We thank P. Deligne for his comments and for indicating us a mistake in an earlier version following an inaccuracy in \cite{ES}\footnote{See \url{https://page.mi.fu-berlin.de/esnault/preprints/helene/138b_small_erratum.pdf} for a correction.}.
We thank an anonymous referee for a thorough reading of this manuscript and for many suggestions that have improved readability.
H. H. is supported by the National
Natural Science Foundation of China (Grants No. 12471012) and the Natural Science Foundation of Jiangsu Province (Grant No. BK20231539).
J.-B. Teyssier is supported by the Institut Universitaire de France (IUF).


\subsection*{Linear Overview}
Section 2 provides an account of Abbes and Saito theory for the logarithmic conductor.
In section 3,  we introduce Beilinson's singular support and Saito's characteristic cycle for étale sheaves and recall some terminology from \cite{HuTeyssierCohBoundedness} to state the cohomological boundedness result from  \cite{HuTeyssierMinkowski} we need.
In section 4,  we give the cohomological interpretation of the total degree for the characteristic cycle (\cref{thm_7}) and deduce \cref{thm_8}.
A consequence is the existence of some dense open set of hyperplanes achieving transversality for perverse sheaves with bounded rank and log conductors (\cref{micro_local_Lefschetz_hyperplane}).
In section 5, we give a criterion for the restriction to a hyperplane section to distinguish two sheaves in terms of the singular support of their $\cHom$-sheaf (\cref{transversality_and_Hom}).
The wild Lefschetz theorem for finite coefficients  (\cref{thm_5}) is then deduced.
Section 6  provides the basic language for moduli of multi-flags used to formulate Lefschetz type results  in codimension $>1$,  thus giving some flexibility for future applications.
In section 7, we derive the existence of some dense open subset of a moduli of multi-flags whose points achieve simultaneous recognition for all pair of perverse sheaves with bounded rank and log conductors (\cref{Lefschetz_recognition}) and extend \cref{thm_5} and \cref{thm_8}  in higher codimension. See \cref{Wild_Lefschetz} and \cref{micro_local_Lefschetz} respectively.

\subsection*{Relation with \cite{HuTeyssierSemicontinuity,HuTeyssierCohBoundedness,HuTeyssierMinkowski}}
The present paper is the last of a series of four papers.
Its goal is to derive three more applications  from \cite{HuTeyssierMinkowski}, namely \cref{Bounddmdeg_normal}, \cref{Lefschetz_recognition} and \cref{Wild_Lefschetz}.
It can be read pretty much independently from \cite{HuTeyssierSemicontinuity,HuTeyssierCohBoundedness,HuTeyssierMinkowski}.
Indeed on the one hand, it relies on  \cite{HuTeyssierCohBoundedness} only via \cref{bounded_conductor}, which is needed to formulate the results for perverse sheaves instead of local systems.
On the other hand,  \cite{HuTeyssierMinkowski} is used as a black-box via one single result,  namely \cref{general_boundedness_complex}.

\begin{notation}
We introduce the following running notations.
\begin{enumerate}\itemsep=0.2cm
\item[$\bullet$] $k$  denotes a perfect field of characteristic $p>0$.
\item[$\bullet$] The letter $\Lambda$ will refer to a finite local ring of residue characteristic $\ell \neq p$.

\item[$\bullet$] For a scheme $X$ of finite type over $k$, we denote by $D_{ctf}^b(X,\Lambda)$ the derived category of complexes of $\Lambda$-sheaves of finite tor-dimension with bounded and constructible cohomology sheaves.

\item[$\bullet$]
 $\LC_{tf}(X,\Lambda)$ will denote the category of locally constant constructible sheaves of $\Lambda$-modules of finite tor-dimension over $X$. 
By \cite[Lemma 4.4.14]{Wei}, the germs of any $\cL\in \LC_{tf}(X,\Lambda)$ are automatically free  $\Lambda$-modules of finite rank.
 
 \item[$\bullet$]
 $\Perv_{tf}(X,\Lambda)\subset D_{ctf}^b(X,\Lambda)$ will denote the category of perverse sheaves of $\Lambda$-modules of finite tor-dimension over $X$ for the middle perversity function \cite[§2.2]{BBD}.
Recall that if $\Lambda_0$  is the residue field of $\Lambda$, an object $\cK \in D_{ctf}^b(X,\Lambda)$ lies in  $\Perv_{tf}(X,\Lambda)$ if and only if $\Lambda_0 \otimes_{\Lambda} \cK$ is a perverse sheaf.

\item[$\bullet$]  Let  $X$  be a scheme of finite type over $k$ and let $\Lambda$ be a field of characteristic $\neq p$.
For $\cK \in D_{ctf}^b(X,\Lambda)$, we put
$$
\Rk_{\Lambda} \cK := \max\{\rk_{\Lambda}\cH^i \cK_{\xbar}, \text{ where $i\in \bZ$ and $\xbar\to X$ is algebraic geometric}\} \ .
$$
\item[$\bullet$]
For $r\geq 0$, we let $D_{ctf}^{\leq r}(X,\Lambda)\subset D_{ctf}^b(X,\Lambda)$ be the full subcategory spanned by objects $\cK$ such that $\Rk_{\Lambda} \cK\leq r$, and similarly with perverse complexes.

\item[$\bullet$]
For a finite stratification $\Sigma$ of $X$, that is a finite partition of $X$ by locally closed subsets,  we let $D_{\Sigma,tf}^b(X,\Lambda)\subset D_{tf}^b(X,\Lambda)$ be the full subcategory spanned by $\Sigma$-constructible complexes, and similarly with perverse complexes.
\end{enumerate}

\end{notation}

\section{Conductors of étale sheaves}

\subsection{Ramification filtrations}\label{local_fields_notation}

Let $K$ be a henselian discrete valuation field over $k$.
Let $\sO_K$ be the ring of integer of $K$, let $\fm_K$ be the maximal ideal of $\sO_K$ and $F$ the residue field of $\sO_K$. 
Fix $K^{\sep} \supset K$ a separable closure of $K$ and let $G_K$ be the Galois group of $K^{\sep}$ over $K$.
Let  $I_K\subset G_K$ be the inertia subgroup and let  $P_K\subset I_K$ be the wild ramification subgroup.

\begin{recollection}\label{local_fields}

 In \cite{RamImperfect},  Abbes and Saito defined two decreasing filtrations $\{G^r_K\}_{r\in\bQ_{>0}}$ and $\{G^r_{K,\log}\}_{r\in\bQ_{\geq 0}}$ on $G_K$ by closed normal subgroups.   
They are called the {\it the ramification filtration} and {\it the logarithmic ramification filtration} respectively.  
For $r\in\bQ_{\geq 0}$, put
\begin{equation*}
G^{r+}_K=\overline{\bigcup_{s>r}G_K^s}\ \ \ \textrm{and}\ \ \ G^{r+}_{K,\log}=\overline{\bigcup_{s>r}G_{K,\log}^s}.
\end{equation*}

\begin{prop}[{\cite{RamImperfect,as ii,logcc,wr}}]\label{propramfil}
The following properties hold :
\begin{enumerate}
\item
For any $0<r\leq 1$, we have 
$$
G^r_K=G^0_{K,\log}=I_K \text{ and } G^{1+}_K=G^{0+}_{K,\log}=P_K.
$$
\item
For any $r\in \bQ_{\geq 0}$, we have 
$$
G^{r+1}_K\subseteq G^r_{K,\log}\subseteq G^r_K.
$$
If $F$ is perfect, then for any $r\in \bQ_{\geq 0}$, we have 
$$
G^r_{K,\cl}=G^r_{K,\log}=G^{r+1}_K.
$$ 
where $G^r_{K,\cl}$ is the classical wild ramification subgroup as defined in \cite{CL}.
\item
For any $r\in \bQ_{> 0}$, the graded piece $G^r_{K,\log}/G^{r+}_{K,\log}$ is abelian, $p$-torsion and contained in the center of $P_K/G^{r+}_{K,\log}$.
\item
For any $r\in \bQ_{> 1}$, the graded piece $G^r_{K}/G^{r+}_{K}$ is abelian, $p$-torsion and contained in the center of $P_K/G^{r+}_{K}$.
\end{enumerate}
\end{prop}

Let $M$ be a finitely generated $\Lambda$-module with a continuous $P_K$-action. 
 The module $M$ has decompositions 
 \begin{equation}\label{twodecomp}
M=\bigoplus_{r\geq 1}M^{(r)}\ \ \ \textrm{and}\ \ \ M=\bigoplus_{r\geq 0}M_{\log}^{(r)}
\end{equation}
into $P_K$-stable $\Lambda$-submodules where $M^{(1)}=M^{(0)}_{\log}=M^{P_K}$, and such that for every $r\in \bQ_{>0}$,
\begin{align*}
(M^{(r+1)})^{G^{r+1}_K}& =0  \text{ and }    (M^{(r+1)})^{G^{(r+1)+}_K} =M^{(r+1)};\\
(M^{(r)}_{\log})^{G^{r}_{K,\log}}& =0  \text{ and }    (M^{(r)}_{\log})^{G^{r+}_{K,\log}}  =M^{(r)}_{\log}.
\end{align*}
The decompositions \eqref{twodecomp} are respectively called  the {\it slope decomposition} and the {\it logarithmic slope decomposition} of $M$. 
The values $r$ for which $M^{(r)}\neq 0$ (resp. $M^{(r)}_{\log}\neq 0$) are the {\it slopes} (resp. the {\it logarithmic slopes}) of $M$.
We denote by $c_K(M)$ the largest slope of $M$ and refer to $c_K(M)$ as the \textit{conductor of $M$}.
Similarly, we denote by $\lc_K(M)$ the largest logarithmic slope of $M$ and refer to $\lc_K(M)$ as the \textit{logarithmic conductor of $M$}.
We say that $M$ is {\it isoclinic} (resp. {\it logarithmic isoclinic}) if $M$ has only one slope (resp. only one logarithmic slope).

The following is an immediate consequence of  \cref{propramfil}-(2).

\begin{lem}\label{inequalityLogNonLog}
 Let $M$ be a finitely generated $\Lambda$-module with a continuous $P_K$-action. 
Then, 
$$
\lc_{K}(M)  \leq c_K(M)\leq  \lc_{K}(M)+1. 
$$
\end{lem}

If $M$ is free as a $\Lambda$-module, then so are the $M^{(r)}_{\log}$ and the $M^{(r)}$ in virtue of \cite[Lemma 1.5]{KatzGauss}.
In that case, the {\it total dimension} of $M$ is defined by
$$
\dt_K(M):=\sum_{r\geq 1}r\cdot\rk_{\Lambda} M^{(r)}
$$
and the {\it Swan conductor} of $M$ is defined by
$$
\sw_K(M):=\sum_{r\geq 0}r\cdot\rk_{\Lambda} M^{(r)}_{\log} \ .
$$

\begin{lem}[{\cite{RamImperfect}}]\label{inequality_Swan_dimtot}
In the setting of \cref{local_fields}, we have
$$
 \sw_K(M) \leq  \dt_K(M) \leq \sw_K(M)+\rk_{\Lambda}M \ .
$$
If the residue field $F$ is perfect, we have
\begin{align*}
\lc_{K}(M)+1&=c_K(M) \ .  \\
\sw_K(M)+\rk_{\Lambda}M&= \dt_K(M) \ .
\end{align*}
\end{lem}

\end{recollection}

\subsection{Conductor divisors}\label{semi_continuity_conductors}\label{semi_continuity_section}
Let $X$ be a normal scheme of finite type over $k$.
Let $Z$ be an integral Weil divisor and let $\eta\in Z$ be its generic point.
Let $K$ be the fraction field of $\hat{\mathcal{O}}_{X,\eta}$ and fix a separable closure $K^{\sep}$ of $K$.
For  $\cF\in \Cons(X,\Lambda)$, the pull-back $\cF|_{\Spec K}$ is a $\Lambda$-module of finite type with continuous $G_{K}$-action.
Using the notations from \cref{local_fields_notation}, we put
$$
c_{Z}(\cF):= c_{K}(\cF|_{\Spec K})  \text{ and }
 \lc_{Z}(\cF):= \lc_{K}(\cF|_{\Spec K})  \ .
$$

\begin{definition}\label{def_LC}
Let $X$ be a normal scheme of finite type over $k$ and let $\cF\in \Cons_{tf}(X,\Lambda)$.
We define the \textit{conductor divisor  of $\cF$} as the Weil divisor with rational coefficients given by
$$
C_X(\cF):=\sum_{Z}  c_{Z}(\cF) \cdot  Z
$$
and the \textit{logarithmic conductor divisor  of  $\cF$} as the Weil divisor with rational coefficients given by
$$
LC_X(\cF):=\sum_{Z} \lc_{Z}(\cF) \cdot  Z
$$
where the sums run over the set of integral Weil divisors of $X$.
\end{definition}

\begin{rem}
The above divisors are $\mathds{Q}$-Weil divisors of $X$.
We will sometimes abuse the notations and write $C(\cF)$ instead of $C_X(\cF)$ and similarly in the logarithmic case.
\end{rem}

\begin{definition}
In the setting of \cref{semi_continuity_section} and for every Weil divisor $D$,  we define the \textit{generic conductor} and the \textit{generic logarithmic conductor of $\cL$ along $D$} respectively by
$$
c_D(\cL):= \max_{Z} c_{Z}(\cL) \text{ and }
 \lc_D(\cL):= \max_{Z}  \lc_{Z}(\cL)   \ .
$$
where $Z$ runs over the set of irreducible components of $D$.
\end{definition}

\section{Singular support and characteristic cycle of étale sheaves}

\subsection{The singular support}\label{Sing_support_setting}
Let $X$ be a smooth  scheme of finite type over $k$. 
We denote by $\bT^*X$ the cotangent bundle of $X$.
Let $C\subset \bT^*X$ be a closed conical subset. 
For a point $x\in X$, we put $\bT^*_{x}X=\bT^*X\times_X x$   and $C_{ x}=C\times_X x$.  

\begin{recollection}
Let $h:U\to X$ be a morphism of smooth schemes of finite type over $k$. 
For $u\in U$, we say that $h:U \to X$ is $C$-{\it transversal at $u$} if 
$$
\ker dh_{u} \bigcap C_{h(u)}\subseteq \{0\}\subseteq
\bT^*_{h(u)}X 
$$ 
where $dh_{u}:\bT^*_{h(u)}X\to  \bT^*_{u}U$ is the cotangent map of $h$ at $u$. 
We say that $h:U \to X$ is \textit{$C$-transversal} if it is $C$-transversal at every point of $U$. 
For a $C$-transversal morphism $h:U\to X$, we let $h^\circ C$ be the scheme theoretic image of $C\times_XU$ in $\bT^*U$ by  $dh:\bT^*X\times_XU \to \bT^*U$.\\ \indent
Let $f:X\to Y$ be a morphism of smooth schemes of finite type over $k$.  
For $x\in X$, we say that $f:X \to  Y$ is $C$-{\it transversal at} $x$ if 
$$
df_{x}^{-1}(C_{x})\subseteq \{0\}\subseteq \bT^*_{f(x)}Y
$$
We say that $f:X\to Y$ is $C$-{\it transversal } if it is $C$-transversal at every point of $X$.\\ \indent
Let $(h,f):Y\leftarrow U\to X$ be a pair of morphisms of between smooth schemes of finite type over $k$. 
We say that $(h,f)$ is $C$-{\it transversal} if $h:U \to X$ is $C$-transversal and if $f:U \to  Y$ is $h^\circ C$-transversal. 
\end{recollection}


\begin{definition}
In the setting of \cref{Sing_support_setting}, we say that $\cK\in D^b_c(X,\Lambda)$ is {\it micro-supported on} $C$ if for every $C$-transversal pair $(h,f):Y\leftarrow U \to   X$, the map $f:U \to  Y$ is universally locally acyclic with respect to $h^*\cK$.
\end{definition}

\begin{theorem}[{\cite[Theorem 1.3]{bei}}]\label{beilinson_theorem}
For every  $\cK\in D^b_c(X,\Lambda)$, there is a smallest closed conical subset $SS(\cK)\subset \bT^*X$ on which $\cK$ is micro-supported. 
Furthermore, if $X$ has pure dimension $n$, then $SS(\cK)$ has pure dimension $n$.
\end{theorem}

\begin{definition}
The closed conical subset $SS(\cK)$ is the \textit{singular support of $\cK$}. 
\end{definition}

\subsection{Base change and transversality}
Let $f:Y\to X$ be a separated morphism of smooth schemes of finite over $k$.
Following \cite[§ 8.2]{cc}, for every $\cK\in D^b_c(X,\Lambda)$, there is a canonical morphism
\begin{equation}\label{Ftransversality}
h_{f,\cK}: f^*\cK\otimes_{\Lambda}^L Rf^!\Lambda\to Rf^!\cK
\end{equation}
obtained by adjunction from the composition
\begin{equation*}
Rf_!(f^*\cK\otimes^L_{\Lambda}Rf^!\Lambda)\xrightarrow{\sim} \cK\otimes^L_{\Lambda}Rf_!Rf^!\Lambda\to\cK
\end{equation*}
where the first arrow is the projection formula and where the second arrow is induced by the adjunction $Rf_!Rf^!\Lambda\to  \Lambda$.
Let us recall the following 

\begin{prop}[{\cite[Proposition 8.13]{cc}}]\label{upper_shriek_transversality}
If $\cK$ is of finite tor-dimension and if $f:Y\to X$ is $SS(\cK)$-transversal, then the canonical morphism \eqref{Ftransversality} is an isomorphism.
\end{prop}

\begin{defin}\label{cf}
For a morphism $f : Y\to X$ between pure dimensional smooth schemes of finite type over  $k$, we put $c_f = \dim Y-\dim X$.
\end{defin}

The following proposition generalizes \cite[Corollary 2.13]{HT}.

\begin{prop}[{\cite[Proposition 1.1.8]{saito22}}]\label{basechange}
Let
$$
	\begin{tikzcd}
V  \drar[phantom, "\scalebox{0.8}{$\square$}"]\arrow{r}{f'}&U\arrow{d}{g}\\
Y\arrow[leftarrow, u, "g' "]\arrow{r}{f}&X
\end{tikzcd}
$$
be a cartesian diagram of smooth schemes of finite type over $k$ of pure dimension.
 Let $\cK\in D^b_{ctf}(U,\Lambda)$ and assume that the following hold :
\begin{enumerate}\itemsep=0.2cm
\item The morphism $f : Y \to X$ is separated.
\item We have $c_f=c_{f'}$.
\item $f' : V\to U$ is $SS(\cK)$-transversal.
\end{enumerate}
Then, $f : Y\to X$ is $SS(Rg_*\cK)$-transversal if and only if the base change morphism
$$
f^*Rg_*\cK \to  Rg'_*f'^*\cK
$$
is an isomorphism.
\end{prop}

\begin{proof}
Since all schemes are smooth over $k$, the two complexes $Rf^!\Lambda$ and $Rf'^!\Lambda$ have locally constant cohomologies by the Poincar\'e duality.
Since $c_f=c_{f'}$ and $V=Y\times_XU$ is smooth, the base change morphism $g'^*f^!\Lambda\to f'^!\Lambda$ is an isomorphism (\cite[Lemma 1.1.4]{saito22}). 
Then, \cref{basechange} is a direct consequence of  \cref{upper_shriek_transversality} and \cite[Proposition 1.1.8]{saito22}.
\end{proof}


\subsection{The characteristic cycle}
Let $f:X \to   S$ be a morphism between smooth schemes of finite type over $k$ where $S$ is a curve over $k$.
Let $x\in X$ be a closed point  and put $s=f(x)$. 
Note that any local trivialization of $\bT^*S$ in a neighborhood of $s$ gives rise to a local section of $\bT^*X$ in a neighborhood of $x$ by applying $df:\bT^*S\times_S X \to \bT^*X$. 
We abusively denote by $df$ this section. \\ \indent
We say that $x$ is an {\it at most isolated $C$-characteristic point for $f:X \to   S$} if there is an open neighbourhood $U\subset X$ of $x$ such that  $f:U- \{x\} \to   S$ is $C$-transversal. 
In that case, the intersection of a cycle $A$ supported on $C$ with $[df]$ is supported at most at a single point in $\bT^*_xX$. 
Since $C$ is conical, the intersection number $(A, [df])_{\bT^*X,x}$ is independent of the chosen local trivialization for $\bT^*S$ in a neighborhood of $s$. 

\begin{theorem}[{\cite[Theorem 5.9]{cc}}]\label{Milnor_formula}
Let $X$ be a smooth scheme of finite type over $k$.
For every  $\cK\in D^b_{ctf}(X,\Lambda)$, there is a unique cycle $CC(\cK)$ of $\bT^{\ast}X$ supported on $SS(\cK)$  such that for every \'etale morphism $h: U \to  X$, for every morphism $f: U \to   S$ with $S$ a smooth curve over $k$,  for every at most isolated $h^{\circ}(SS(\cK))$-characteristic point $u\in U$ for $f:U \to   S$, we have the following Milnor type formula
 \begin{equation}\label{Milnor}
 -\dt (R\Phi_{\ubar}(h^{\ast}\cK, f))=(h^*CC(\cK),[df])_{T^{\ast}U, u}
 \end{equation}
 where $R\Phi_{\ubar}(h^{\ast}\cK, f)$ denotes the stalk of the vanishing cycle of $h^*\cK$ with respect to $f:U \to  S$ at a geometric point $\ubar \to   U$ above $u$.
\end{theorem}

\begin{definition}
The cycle $CC(\cK)$ from \cref{Milnor_formula} is the {\it characteristic cycle of $\cK$}. 
\end{definition}

\begin{rem}[{\cite[Proposition 5.14]{cc}}]\label{CC_perv}
When $\cK$ is perverse, $SS(\cK)$ and $CC(\cK)$ have the same support and the multiplicities of $CC(\cK)$ are positive integers.
\end{rem}

We store for future use the following 

\begin{lem}\label{automatic_transversality}
Let $X$ be a scheme of finite type over a field $k$.
Let $j : U\hookrightarrow X$ be an affine open immersion with  $U$  smooth over $k$.
Let $i : X \hookrightarrow Y$ be a closed immersion in a smooth scheme of finite type over $k$.
Then, for every $\cP\in Perv_{\ft}(U,\Lambda)$, we have 
$$
SS(i_* j_! \cP)= SS(i_* Rj_* \cP) \ .
$$
\end{lem}

\begin{proof}
Since $j : U\hookrightarrow X$ is an affine open immersion, the complexes $i_* j_! \cP$ and $i_* Rj_* \cP$ are perverse by \cite[Corollaire 4.1.10]{BBD}.
Hence $SS(i_* j_! \cP) =\Supp CC(i_* j_! \cP)$
and $SS(i_* Rj_* \cP) =\Supp CC(i_* Rj_* \cP)$ by \cref{CC_perv}.
Since $CC(i_* j_! \cP) =  CC(i_* Rj_* \cP)$ by \cite[Lemma 5.13-(3)]{cc}, the conclusion follows.
\end{proof}

The following index formula provides a positive characteristic analogue of Kashiwara-Dubson's formula for $\mathcal{D}$-modules.

\begin{theorem}[{\cite[Theorem 7.13]{cc}}]\label{CC_and_chi}
Let $X$ be a smooth projective variety over an algebraically closed field $k$.
For every $\cK\in D^b_{ctf}(X,\Lambda)$, we have 
$$
\chi(X,\cK)= (CC(\cK), \bT^*_{X} X)_{\bT^*X} \ .
$$
\end{theorem}


\subsection{Bounding the ramification with coherent sheaves}\label{bounded_ramification_section}

Let $X$ be a scheme of finite type over $k$.
We denote by $\bQ[\Coh(X)]$ the free $\bQ$-vector space on the set of isomorphism classes of coherent sheaves on $X$.
Observe that the pullback along every morphism $f : Y\to X$ of schemes of finite type over $k$ induces a morphism of $\bQ$-vector spaces
$$
f^* : \bQ[\Coh(X)]\to \bQ[\Coh(Y)]\ .
$$
Assume now that $X$ is normal and let $\cE\in \Coh(X)$.
If $X^1\subset X $ denotes the set of codimension $1$ points of $X$, 
we define a Weil divisor on $X$ by the formula
$$
T(\cE):= \sum_{\eta\in X^1}    \length_{\cO_{X,\eta}}( \cE|_{X_{\eta}}^{\tors})\cdot \overline{\{\eta\}}
$$
where $X_{\eta}= \Spec \cO_{X,\eta}$ and where $\cE|_{X_{\eta}}^{\tors}$  is the torsion part of $\cE|_{X_{\eta}}$.

\begin{example}
If $R$ is an effective Cartier divisor of $X$ with ideal sheaf $\cI_R$ and if $\cE=\cO_X/\cI_R$, then $T(\cE)=R$.
\end{example}

If $\Weil(X)_{\bQ}$ is the space of $\bQ$-Weil divisors on $X$, the map $T : \Coh(X)\to \Weil(X)_{\bQ}$ induces a map of $\bQ$-vector spaces
$$
T : \bQ[\Coh(X)]\to \Weil(X)_{\bQ} \ .
$$

\begin{definition}\label{bounded_conductor}
Let $X$ be a scheme of finite type over $k$.
Let $\cK\in D_c^b(X,\Lambda)$ and  $\cE\in \bQ[\Coh(X)]$.
We say that \textit{$\cK$ has log conductors bounded by $\cE$} if for every morphism $f : C\to X$ over $k$ where $C$ is a smooth curve over $k$, we have
$$
LC(\cH^i\cK|_C)\leq T(f^*\cE) 
$$
for every $i\in \bZ$ where the left-hand side is defined in \cref{def_LC}.
We denote by $D_c^b(X,\cE,\Lambda)$ the full subcategory of $D_c^b(X,\Lambda)$ spanned by objects having log conductors bounded by $\cE$.
\end{definition}

The following is our main example of sheaf with explicit bound on the log conductors.

\begin{prop}[{\cite[Proposition 5.8]{HuTeyssierCohBoundedness}}]\label{bounded_ramification_ex}
Let $X$ be a normal scheme of finite type over $k$.
Let $D$ be an effective Cartier divisor of $X$ and put $j : U:=X-D\hookrightarrow X$.
Let  $\cL\in \LC_{\ft}(U,\Lambda)$ and $\cE\in \bQ[\Coh(X)]$.
\begin{enumerate}\itemsep=0.2cm
\item If $j_!\cL$ has log conductors bounded by $\cE$, then $LC_X(j_!\cL) \leq T(\cE)$.
\item If $X$ is smooth over $k$, then $j_!\cL$ has log conductors bounded by $(\lc_D(\cL)+1)\cdot \cO_D$.
\end{enumerate}

\end{prop}

\begin{lem}\label{ramification _of_Hom}
Let $X$ be a scheme of finite type over $k$ and let $\cE\in \bQ[\Coh(X)]$.
Let $j : U\hookrightarrow X$ be an open immersion and let $\cL_1,\cL_2\in \Loc_{\ft}(U,\cE,\Lambda)$.
Then, $\cHom(\cL_1,\cL_2) \in \Loc_{\ft}(U,\cE,\Lambda)$.
\end{lem}

\begin{proof}
Since $\cL_1,\cL_2$ are locally constant constructible sheaves, the formation of the sheaf $j_!\cHom(\cL_1,\cL_2)$ commutes with pullback.
Hence, we reduce to an analogous statement where $X$ is the spectrum of a strict henselian dvr over $k$, where the statement is obvious.
\end{proof}

\begin{defin}\label{admissible}
Let $X$ be a scheme of finite type over $k$.
We say that a $\bQ$-linear map $\fc : \bQ[\Coh(X)]\to \bQ$ is \textit{admissible} if the following conditions are satisfied :
\begin{enumerate}\itemsep=0.2cm
\item For every $\cE\in \Coh(X)$, we have $\fc(\cE)\in \bN$.
\item For every $\cE_1,\cE_2\in \Coh(X)$, we have $\fc(\cE_1\bigoplus \cE_2)\leq \fc(\cE_1)+\fc(\cE_2)$.
\end{enumerate}
\end{defin}

\begin{defin}
Let $P$ be a property of morphisms of schemes over $k$ and let $n\geq 0$. 
A $P$-relative stratified scheme $(X/S,\Sigma)$ of relative dimension $\leq n$ will refer to a morphism $X\to S$ between schemes of finite type over $k$ satisfying $P$ such that the fibres of $X\to S$ have dimensions $\leq n$ and where $X$ is endowed with a finite stratification $\Sigma$.
\end{defin}

The following theorem is one of the main result of \cite{HuTeyssierMinkowski} :
\begin{thm}[{\cite[Corollary 7.30]{HuTeyssierMinkowski}}]\label{general_boundedness_complex}
Let $(X/S,\Sigma)$ be a proper relative  stratified scheme of relative dimension $\leq n$ and let $a\leq b$ be integers.
Then, there is an admissible function $\fc : \bQ[\Coh(X)]\to \bQ$ and  $\fd\in \mathds{N}[x]$ of degree $n$ such that for every algebraic geometric point $\sbar\to S$, every finite field $\Lambda$ of characteristic $\ell\neq p$,  every $\cE\in \bQ[\Coh(X)]$ and every  $\cK\in D_{\Sigma_{\sbar}}^{[a,b]}(X_{\sbar},\cE_{\sbar},\Lambda)$, we have 
$$
\sum_{j\in \bZ} h^j(X_{\sbar},\cK)\leq \fd(\fc(\cE))\cdot \Rk_{\Lambda}\cK  \ .
$$
\end{thm}

\section{Characteristic cycle boundedness}

\subsection{Geometric situation}\label{geo_situation}
      Let $S$ be a scheme of finite type over  $k$, let $E$ be a locally free sheaf  of $\cO_S$-modules on $S$.
      We let $Q_S$ be the relative universal hyperplane  in $\bP_S(E)$ and $F_S$ the relative universal projective line in $\bP_S(E^{\vee})$.
      Let $f: X\to S$ be a projective morphism over $k$ and let $i :X\hookrightarrow \bP_S(E)$ be a closed immersion.
Consider the commutative diagram with cartesian squares
$$
\begin{tikzcd}  
X_{L}\ar{r}   \ar{d}   & ( X\times_{\bP_S(E)} Q_S ) \times_{\bP_S(E^{\vee})} F_S   \ar{r}  \ar{d} \arrow[rightarrow, rr,bend left = 25,  "q_X"]   & X\times_{ \bP_S(E)} Q_S  \ar{r}{p_X} \ar{d}  &     X \ar{d}{i} \\
           \bP(E)_L  \ar{r}  \ar{d}{p^{\vee}_L}           &         F_S \times_{\bP_S(E^{\vee}) } Q_S \ar{r}  \ar{d}        &         Q_S      \ar{r}{p}  \ar{d}{p^{\vee}}     &\bP_S(E)   \\
L \ar{r}    \ar{d}      &  F_S \ar{r}  \ar{d}  & \bP_S(E^{\vee}) \arrow[leftarrow, uu, crossing over ,bend left = 50, near end,  "p_X^{\vee}"]  &  \\
\{L\}   \ar{r}             &    \bG_S(E^{\vee},2)     \arrow[leftarrow, uuu, crossing over ,bend left = 70, near end,  "q_X^{\vee}"]             & 
	\end{tikzcd}
$$
where $L\in  \bG_S(E^{\vee},2)$ is a pencil lying above an algebraic geometric point $\sbar\to S$ lying over $s\in S$.
        Let $\overline{\eta}_s\to \bP_S(E^{\vee})$ be an algebraic geometric point over the generic point of $\bP_{\sbar}(E_{\sbar}^{\vee})$.
       Let  $\overline{\xi}_s\to  \bG_S(E^{\vee},2)  $ be an algebraic  geometric point over  the generic point of $ \bG_{\sbar}(E_{\sbar}^{\vee},2)$.
We consider the cartesian squares
$$
	\begin{tikzcd}
		X_{\overline{\eta}_{s}}\arrow{r}{ } \arrow{d} & X\times_{ \bP_S(E)} Q_S  \arrow{d}{p_X^{\vee}} \\
		\overline{\eta}_s \arrow{r}{ } & \bP_S(E^{\vee})
	\end{tikzcd}
	\qquad 
	\begin{tikzcd}
		X_{\overline{\xi}_s} \arrow{r} \arrow{d}&  ( X\times_{ \bP_S(E)} Q_S ) \times_{\bP_S(E^{\vee})} F_S \arrow{d}{q_{X}^{\vee} } \\
		\overline{\xi}_s\arrow{r} & \bG_S(E^{\vee},2)  .
	\end{tikzcd}
$$

\begin{rem}
We think about $X_{\overline{\eta}_s}$ as the generic hyperplane section of $X_{\sbar}$  and about $X_{\overline{\xi}_s}$ as the generic pencil of $X_{\sbar}$.
\end{rem}

\begin{rem}\label{genericXs}
Assume that $S=\Spec k$, in which case we drop the subscript $S$, let $\Lambda$ be a finite field of characteristic $\ell \neq p$ and let $\cK \in D_c^b(X,\Lambda)$.
Since $
 \chi(X_{\overline{\eta}}, \cK|_{X_{\overline{\eta}}} )$ is the generic rank of $Rp_{X\ast}^{\vee}(\cK|_{ X\times_{ \bP(E)} Q  })$, we have 
$$
 \chi(X_{\overline{\eta}}, \cK|_{X_{\overline{\eta}}}) =\chi(X\cap H, \cK|_{X\cap H})
$$ 
for every  sufficiently generic hyperplan $H\subset \bP(E)$.
Similarly, since $\chi(X_{\overline{\xi}}, \cK|_{X_{\overline{\xi}}})$ is the generic rank of $Rq_{X\ast}^{\vee}(\cK|_{( X\times_{ \bP(E)} Q ) \times_{\bP(E^{\vee})} F })$, we have 
$$
\chi(X_{\overline{\xi}}, \cK|_{X_{\overline{\xi}}}) =\chi(X_{L}, \cK|_{X_L})
$$ 
for every  sufficiently generic pencil $L\subset \bP(E^{\vee})$.
\end{rem}

\begin{assumption}[{\cite[4.3]{bei}}]\label{relative_beilinson_assumption}
Let $S$ be a scheme of finite type over $k$.
Let $E$ be a locally free sheaf of $\cO_S$-modules on $S$.
Let $X$ be a smooth projective scheme over $S$,  let $i : X\hookrightarrow \bP_S(E)$ be a closed immersion and put $\cL := i^*\cO_{\bP_S(E)}(1)$.
For every algebraic geometric point $\sbar\to S$ and  any distinct closed points $u,v\in X_{\sbar}$,  the composition 
$$
E_{\sbar} \to \Gamma(X_{\sbar},\cL|_{X_{\sbar}})  \to \cL_u/\mathfrak{m}^2_u \cL_u \oplus \cL_v/\mathfrak{m}^2_v \cL_v
$$
is surjective.
\end{assumption}

\subsection{Recollection}
In the setting of \cref{geo_situation}, assume that  $X$ is a projective space $\bP$ over $k$ and let $L\subset \bP^{\vee}$ be a pencil with axis  $A\subset \bP$.
In that case, the diagram from \cref{geo_situation} simplifies into  the following cartesian diagram 
$$
\begin{tikzcd}    
\bP_L \ar{r}   \arrow[rr,  bend left = 40, "\pi"] \arrow{d}{p^{\vee}_L} & Q\ar{d}{p^{\vee}}  \ar{r}{p} & \bP\\
L \ar{r}&  \bP^{\vee}& 
	\end{tikzcd}
$$
where  $\pi : \bP_L \to \bP$ is the blow-up of $\bP$ along $A$.
We have 
$$
\bP_L^{\circ}:=\bP_L -\pi^{-1}(A)  = \bP-A \subset Q-\bP(\bT^*_A \bP)
$$
with $\pi^{-1}(A)$ identifying to $\bP(\bT^*_A \bP)$ via the closed immersion $\bP_L \to Q$.
For a closed conical subset $C\subset \bT^* \bP$, we put 
$$
Z_L(C):=\bP_L  \cap \bP(C)\subset \bP_L \ .
$$ 

\begin{lem}[{\cite[Lemma 3.10]{cc}}]\label{transversality_H}
Let $C\subset \bP$ be a closed conical subset of pure dimension $\dim \bP$.
The complement $Q-\bP(C)$ is the largest open subset such that $p^{\vee} : Q\to \bP^{\vee}$ is $p^{\circ}(C)$-transversal. 
In particular for every  $H\in \bP^{\vee}$, the following are equivalent :
\begin{enumerate}\itemsep=0.2cm
\item the inclusion $H \hookrightarrow \bP$ is $C$-transversal.
\item the map $p^{\vee} : Q\to \bP^{\vee}$ is $p^{\circ}(C)$-transversal at every point $(x,H), x\in H$.
\item $H$ lies in $\bP^{\vee}-p^{\vee}(\bP(C))$.
\end{enumerate}
\end{lem}

\begin{prop}[{\cite[Proposition 4.5]{bei}}]\label{generically_radicial}
Let $X$ be a smooth scheme of finite type of pure dimension $n$ over   $k$.
Let $i :X\hookrightarrow \bP$ be an immersion satisfying \cref{relative_beilinson_assumption}.
Let $C\subset \bT^*X$ be  a closed conical subset of pure dimension $n$ and let $(C_a)_a$ be its irreducible components.
Then
\begin{enumerate}\itemsep=0.2cm
\item If we put 
$$
D_a := \overline{ p^{\vee}(\bP(i_\circ C_a))} \ ,
$$
then $D_a$ is a divisor of $\bP^{\vee}$ and the induced map  $\bP(i_{\circ}C_a) \to D_a$ is generically radicial.
\item If $C_a \neq C_b$, then $D_a \neq D_b$.
\end{enumerate}
\end{prop}

\begin{rem}
As shown in \cite[Theorem 1.7]{bei}, the generic degree of $\bP(i_{\circ}C_a) \to D_a$ is 1 when $p\neq 2$ and 1 or 2 if $p=2$.
\end{rem}

The following  \cref{generic_L} is similar to \cite[Lemma 4.2.7]{UYZ} and \cite[Lemma 2.3]{SY}.

\begin{lem}\label{generic_L}
     Let $X$ be a smooth projective scheme of pure dimension $n$ over $k$ algebraically closed.
     Let $i :X\hookrightarrow \bP$ be a closed immersion satisfying \cref{relative_beilinson_assumption}.
     Let $C\subset \bT^*X$ be  a closed conical subset of pure dimension $n$ 
and put $D = p^{\vee}(\bP(i_{\circ}C))$ endowed with its reduced closed subscheme structure.
      Then, for every sufficiently generic  pencil $L$, we have :
\begin{enumerate}\itemsep=0.2cm
\item The map $\pi : \bP_L \to \bP$ is properly $i_{\circ}C$-transversal.
\item  $Z_L(i_{\circ}C)$ is a subset of $\deg D$ points of $\bP_L^{\circ}$ mapping bijectively to $D\cap L$.
\item the map $p^{\vee}_L : \bP_L\to L$ is $\pi^{\circ}i_{\circ}C$-transversal away from $Z_L(i_{\circ}C)$.
\item the points of $Z_L(i_{\circ}C)$ are isolated $\pi^{\circ}i_{\circ}C$-characteristic points for $p^{\vee}_L : \bP_L\to L$.
\end{enumerate}
\end{lem}
\begin{proof}
Since $C$ has pure dimension $n$,  note that $i_{\circ}C$ has pure dimension $\dim \mathds{P}$.
For a sufficiently generic pencil $L$ with axis $A$,  the inclusion $A\to \bP$ is properly $i_{\circ}C$-transversal by \cite[Lemma 1.3.7-(2)]{saito21}.
Since $\pi : \bP_L \to \bP$ is the blow-up of $\bP$ along $A$, item (1) follows from \cite[Lemma 4.2.4]{UYZ}.
For a sufficiently generic pencil $L$,  the set $D\cap L$ has exactly $\deg D$ elements.
Then (2) follows from \cref{generically_radicial}.
By \cref{transversality_H}, the right vertical arrow of the cartesian square
$$
\begin{tikzcd}    
\bP_L - Z_L(i_{\circ}C) \ar{r}   \arrow{d}{p^{\vee}_L}  & Q- \bP(i_{\circ}C) \ar{d}{p^{\vee}}    \\
L \ar{r}&  \bP^{\vee}                       & 
	\end{tikzcd}
$$
is $p^{\circ} i_{\circ}C$-transversal.
By \cite[Lemma 3.9-(2)]{cc},  we deduce (3).
We now prove (4). 
Choose a sufficiently generic pencil $L$ and pick $(x,H)\in Z_L(i_{\circ}C)$.
Let $U\subset Q$ be an open neighbourhood of $(x,H)$ containing no other point of $Z_L(i_{\circ}C)$ than $(x,H)$.
By (3),  we need to show that $p^{\vee}_L: U\cap \bP_L \to L $ is not $\pi^{\circ}i_{\circ}C$-transversal.
By (1) and \cite[Lemma 3.4-(3)]{cc},  the map $\bP_L\to Q$ is $p^{\circ} i_{\circ}C$-transversal.
Hence by  \cite[Lemma 3.9-(1)]{cc},  it is equivalent to show that $p^{\vee}: Q \to\bP^{\vee} $  is not $p^{\circ} i_{\circ}C$-transversal on a neighbourhood of $U\cap \bP_L$. 
Since $p^{\vee}: Q \to\bP^{\vee} $  is not $p^{\circ} i_{\circ}C$-transversal  at $(x,H)\in \bP(i_{\circ}C)$ in virtue of  \cref{transversality_H} the conclusion follows.

\end{proof}

\begin{rem}
Note that \cref{generic_L}-(4) is slightly stronger than needed for our purpose.
We will just need  that the points of $Z_L(i_{\circ}C)$ are \textit{at most} isolated $\pi^{\circ}i_{\circ}C$-characteristic points for $p^{\vee}_L : \bP_L\to L$.
\end{rem}


\begin{thm}\label{chiXL}
Let $X$ be a smooth projective scheme of pure dimension $n$ over $k$ algebraically closed.
Let $i :X\hookrightarrow \bP$ be a closed immersion in some projective space over $k$ of dimension $m$ satisfying \cref{relative_beilinson_assumption}.
Let $\Lambda$ be a finite field of characteristic $\ell \neq p$, let $\cK\in D_{c}^{b}(X,\Lambda)$ and write 
$$
SS(\cK)  =\bigcup_a C_a  \text{ and } CC(\cK)= \sum_a  m_a\cdot [C_a]
$$ 
where the $C_a \subset \bT^*X$ are distinct closed  irreducible conical subsets.
Put $D_a = p^{\vee}(\bP(i_\circ C_a))$ and let $d_a$ be the generic inseparable degree of $\bP(i_\circ C_a) \to D_a$.
Then,
$$
\chi(X_{\overline{\xi}}, \cK|_{X_{\overline{\xi}}})  =2 \cdot \chi(X_{\overline{\eta}}, \cK|_{X_{\overline{\eta}}} )+(-1)^{n-m} \sum_a d_a \cdot m_a \cdot \deg D_a \ .
$$
\end{thm}
\begin{proof}
By \cite[Lemma 5.13]{cc}, we have
$$
SS(i_*\cK) = \bigcup_a i_{\circ}C_a \text{ and } CC(i_* \cK)= (-1)^{n-m}\cdot \sum_a  m_a\cdot [i_\circ C_a]  \ .
$$
Let $L$ be a  sufficiently generic pencil. 
By \cref{generically_radicial}-(2), the $D_a\cap L$ are two by two disjoint.
Hence, so are the $Z_L(i_{\circ}C_a)$.
In particular, 
$$
Z_L(SS(i_*\cF)) = \bigsqcup_a Z_L(i_{\circ}C_a)  \subset \bP_L^{\circ}
$$
where $Z_L(i_{\circ}C_a)$  consists in exactly $\deg D_a$ points mapping bijectively to $D_a\cap L$ in virtue of \cref{generic_L}-(2).
By \cref{generic_L}-(1) and \cite[Corollary 8.15]{cc},  we have 
\[
SS(\pi^* i_*\cK)=\pi^{\circ}SS(i_*\cK) \ .
\]
Thus,   \cref{generic_L}-(3) implies that  $p_L^{\vee} : \bP_L\to L$ is $SS(\pi^* i_*\cK)$-transversal away from $Z_L(SS(i_*\cK)) $.
Since $\pi^* i_*\cK$ is micro-supported on $SS(\pi^* i_*\cK)$ and since $p_L^{\vee} : \bP_L\to L$ is proper, \cite[Lemma 4.3]{cc} implies that $Rp_{L\ast}^{\vee}\pi^* i_*\cK$  has locally constant cohomology sheaves on $L-\bigcup_a (D_a\bigcap L)$.
Thus
$$
CC(Rp_{L\ast}^{\vee}\pi^* i_*\cK) = -\rk_{\Lambda}(Rp_{L\ast}^{\vee}\pi^* i_*\cK) \cdot [\bT^*_L L] 
                                  +\sum_a \sum_{H\in D_a\bigcap L} n_H \cdot  [\bT^*_H L] \ ,
$$
where the $n_H$ are integers.
By proper base change, we have 
$$
Rp_{L\ast}^{\vee}\pi^* i_*\cK \simeq ( Rp_{X\ast}^{\vee} p_X^*\cK)|_L \ .
$$
For a sufficiently generic pencil $L$, we thus have 
$$
\rk_{\Lambda}(Rp_{L\ast}^{\vee}\pi^* i_*\cK)  = \rk_{\Lambda}(Rp_{X\ast}^{\vee} p_X^*\cK)= \chi(X_{\overline{\eta}}, \cK|_{X_{\overline{\eta}}} ) 
$$
where the last equality holds by  \cref{genericXs}.
By \cref{CC_and_chi}, we deduce that for a sufficiently generic pencil $L$, we have 
$$
\chi(X_{\overline{\xi}}, \cK|_{X_{\overline{\xi}}})  =2 \cdot \chi(X_{\overline{\eta}}, \cK|_{X_{\overline{\eta}}} )+\sum_a \sum_{H\in D_a\bigcap L} n_H \ .
$$
To conclude the proof of \cref{chiXL}, we are thus left to show that for a sufficiently generic pencil $L$, we have $n_H=(-1)^{n-m} \cdot d_a \cdot m_a$ for every  $H\in D_a\bigcap L$.
\cref{Milnor_formula} applied to $\id_L : L\to L$  at $H\in D_a\bigcap L$ yields
$$
n_H=-  \dimtot R\phi_H(Rp_{L\ast}^{\vee}\pi^* i_*\cK, \id_L)  \ .
$$
On the other hand, the compatibility of the vanishing cycles with proper push-forward yields a canonical equivalence 
$$
R\phi_H(Rp_{L\ast}^{\vee}\pi^* i_*\cK, \id_L) \simeq R\Gamma(H\times \{H\}, R\phi(\pi^* i_*\cK, p_{L}^{\vee})) \ .
$$
Let $(x,H)\in Z_L(SS(i_{\circ}C_a))$ be the unique point of $Z_L(SS(i_{\circ}C_a))$ lying over $H$.
Recall that no other point from $Z_L(SS(i_*\cK)) $ lies over $H$.
Hence, $p_L^{\vee} : \bP_L\to L$ is $SS(\pi^* i_*\cK)$-transversal at every point of $H\times \{H\}$ distinct from $(x,H)$.
Thus $p_L^{\vee} : \bP_L\to L$ is  locally acyclic with respect to $\pi^* i_*\cK$ at every point of $H\times \{H\}$ distinct from $(x,H)$.
Hence, the restriction of  $R\phi(\pi^* i_*\cK, p_{L}^{\vee})$ to $H\times \{H\}$ is supported on $(x,H)$.
Thus, 
$$
n_H= - \dimtot R\phi_{(x,H)}(\pi^* i_*\cK, p_{L}^{\vee}) \ .
$$
Since $(x,H)$ is an at most isolated $SS(\pi^* i_*\cK)$-characteristic point for $p_L^{\vee} : \bP_L\to L$ in virtue of \cref{generic_L}-(4), \cref{Milnor_formula}  yields
\begin{align*}
n_H& = (CC(\pi^* i_*\cK), dp_{L}^{\vee})_{\bT^*\bP_L,(x,H)} \\
      & =  (CC(i_*\cK), dp_{L}^{\vee\circ})_{\bT^*\bP,x}  \\
      &=  (-1)^{n-m} \cdot m_a \cdot (i_\circ C_a, dp_{L}^{\vee\circ})_{\bT^*\bP,x} \\
      & = (-1)^{n-m} \cdot  d_a \cdot m_a
\end{align*}
where the second equality follows from the fact that $(x,H)$ lies in  $\bP_L^{\circ}$, where the third equality follows from the fact that  $p_L^{\vee} : \bP_L\to L$ is $i_\circ C_b$-transversal  at $(x,H)$ for every $b\neq a$, and where the last equality follows from  \cite[Lemma 5.4]{cc}.
\end{proof}

\begin{defin}\label{total_degree_CC}
Let $X$ be a smooth projective scheme of pure dimension $n$ over $k$ and let $i : X\hookrightarrow \bP_k^m$ be a closed immersion over $k$ satisfying \cref{relative_beilinson_assumption}.
Let $CC$ be a $n$-cycle of $\bT^*X$ and write
$$
CC= \sum_{a}   m_{a} \cdot [C_{a} ]
$$ 
where the $C_{a}  \subset \bT^*X$ are distinct closed  irreducible conical subsets.
Define $D_{a}  = p^{\vee}(\bP(i_{\circ} C_{a} ))$ and let $d_{a} $ be the generic degree of $\bP(i_{ \circ}  C_{a} ) \to D_{a} $.
We define the \textit{total degree of CC} by
$$
\totdeg_i(CC):=(-1)^{n-m}\sum_{a}  d_{a}  \cdot m_{a}  \cdot \deg D_{a}  \ .
$$
\end{defin}

\begin{rem}\label{bound_deg_SS}
In the setting of \cref{total_degree_CC},  if $CC =CC(\cP)$ where $\cP \in \Perv_{\ft}(X,\Lambda)$, then the $m_a$ are positive and $SS(\cP)$ and $CC(\cP)$ have the same support by \cref{CC_perv}.
Thus, we have 
$$
\deg p^{\vee}(\bP(SS(i_\ast\cP)) \leq | \totdeg_i(CC(\cP)) |\ .
$$
\end{rem}

Before drawing the consequences of the Betti number estimates for the characteristic cycle, let us recall the following 

\begin{lem}[{\cite[Lemma 5.6]{UYZ}}]\label{SS_change_coef}
Let $X$ be a smooth scheme of finite type over $k$.
Let $\Lambda$ be a finite local ring of residue  characteristic $\ell \neq p$ with residue field $\Lambda_0$.
For every $\cK \in D_{ctf}^b(X,\Lambda)$, we have 
$$
SS(\cK)=SS(\cK \otimes_{\Lambda}^L \Lambda_0) \quad \text{and} \quad CC(\cK)=CC(\cK \otimes_{\Lambda}^L \Lambda_0) \ .
$$
\end{lem}

The above lemma suggests to introduce the following

\begin{defin}\label{Rk_ring}
Let $X$ be a scheme of finite type over $k$ and let $\Lambda$ be a finite local ring of residue  characteristic $\ell\neq p$ with residue field $\Lambda_0$. 
For $\cK \in D_{ctf}^b(X,\Lambda)$, we put 
$$
\Rk_{\Lambda}  \cK := \Rk_{\Lambda_0}  \cK\otimes_{\Lambda}^L \Lambda_0 \ .
$$
\end{defin}

\begin{thm}\label{Bounddmdeg_normal}
Let $(X/S,\Sigma)$ be a  relative smooth projective stratified scheme of relative dimension $\leq n$.
Let $i : X\hookrightarrow \bP_S(E)$ be a closed immersion over $S$ satisfying \cref{relative_beilinson_assumption}.
Let $a\leq b$ be integers.
Then, there is an admissible function $\fc : \bQ[\Coh(X)]\to \bQ$ and  $\fd\in \mathds{N}[x]$ of degree $n$ such that for every finite local ring $\Lambda$ of residue characteristic $\ell \neq p$, every algebraic geometric point $\sbar\to S$, every $\cE\in \bQ[\Coh(X)]$ and every $\cK\in D_{\Sigma_{\sbar}, \ft}^{[a,b]}(X_{\sbar},\cE_{\sbar},\Lambda)$,  we have
\begin{equation}\label{eq_Bounddmdeg_normal}
|\totdeg_{i_{\sbar}}(CC(\cK)) | \leq \fd(\fc(\cE)) \cdot \Rk_{\Lambda} \cK \ .
\end{equation}
In particular, if $\cK$ is perverse, the number of irreducible components of $SS(\cK)$ and the multiplicities of $CC(\cK)$ are smaller than $\fd(\fc(\cE)) \cdot \Rk_{\Lambda} \cK$.
\end{thm}
\begin{proof}
The final claim is a consequence of \eqref{eq_Bounddmdeg_normal}  and \cref{CC_perv}.
By \cref{SS_change_coef} and by \cref{Rk_ring}, both sides of \eqref{eq_Bounddmdeg_normal} depend on $\cK$ only via its reduction to the residue field of $\Lambda$.
Hence, we can suppose that $\Lambda$ is a finite field  of characteristic $\ell \neq p$.
By \cref{chiXL}, we have 
$$
|\totdeg_{i_{\sbar}}(CC(\cK)) |\leq |\chi(X_{\overline{\xi}_{s}},\cK|_{X_{\overline{\xi}_{s}}})| + 2\cdot |\chi(X_{\overline{\eta}_{s}},\cK|_{X_{\overline{\eta}_{s}}})|
$$
where $\overline{\eta}_{s}\to \bP_{\sbar}(E_{\sbar}^{\vee})$ is an algebraic geometric point  over the generic point of $\bP_{\sbar}(E_{\sbar}^{\vee})$ and where $\overline{\xi}_s\to  \bG_S(E^{\vee},2)  $ is an algebraic  geometric point  over  the generic point of $ \bG_{\sbar}(E_{\sbar}^{\vee},2)$.
By \cref{general_boundedness_complex} applied to the proper morphism
$$
p_X^{\vee}: X\times_{ \bP_S(E)} Q_S  \to \bP_S(E^{\vee})
$$
of relative dimension $\leq n-1$ endowed with $p_X^*\Sigma$, there is an admissible function $\fc_1 : \bQ[\Coh(X\times_{ \bP_S(E)} Q_S)]\to \bQ$ and $\fd_1\in \mathds{N}[x]$ of degree $n-1$  independent of $\sbar$ and 
$\cK$ such that
$$
|\chi(X_{\overline{\eta}_{s}},\cK|_{X_{\overline{\eta}_{s}}})|\leq 
\fd_1(\fc_1(p_X^*\cE))\cdot \Rk_{\Lambda}\cK \ .
$$
By \cref{general_boundedness_complex} applied to the proper morphism
$$
q_{X}^{\vee}:  ( X\times_{ \bP_S(E)} Q_S ) \times_{\bP_S(E^{\vee})} F_S \to \bG_S(E^{\vee},2) 
$$
of relative dimension $\leq n$ endowed with $q_X^*\Sigma$, there is an admissible function $\fc_2 : \bQ[\Coh( ( X\times_{ \bP_S(E)} Q_S ) \times_{\bP_S(E^{\vee})} F_S)]\to \bQ$ and $\fd_2\in \mathds{N}[x]$ of degree $n$ independent of $\sbar$ and $\cK$ such that
$$
|\chi(X_{\overline{\xi}_{s}},\cK|_{X_{\overline{\xi}_{s}}})|\leq \fd_2(\fc_2(q_X^*(\cE)))\cdot \Rk_{\Lambda}\cK  \ .
$$
The conclusion thus follows by putting $\fc:=\fc_1\circ p_X^*+\fc_2\circ q_X^*$ and $\fd:=2\fd_1+\fd_2$.  
\end{proof}

Similarly to \cite[Theorem 7.34]{HuTeyssierMinkowski},  in the absolute setting when $\cK$ is the extension by 0 of a locally constant constructible sheaf, \cref{Bounddmdeg_normal} admits the following slightly simpler formulation:

\begin{thm}\label{Bounddmdeg_smooth}
Let $X$ be a smooth projective scheme  of pure dimension $n$ over  $k$.
Let $D\subset X$ be an effective Cartier divisor and put $j : U:=X-D\hookrightarrow X$.
Let $i : X\hookrightarrow \bP_k^m$ be a closed immersion over $k$ satisfying \cref{relative_beilinson_assumption}.
Then, there is  $\fd\in \mathds{N}[x]$ of degree $n$ such that for every finite local ring $\Lambda$  of residue characteristic $\ell\neq p$  and every  $\cL\in \Loc_{\ft}(U,\Lambda)$, we have 
$$
| \totdeg_i(CC(j_!\cL)) | \leq \fd(\lc_D(\cL))\cdot \rk_{\Lambda} \cL \ . 
$$
In particular,  the number of irreducible components of $SS(j_!\cL)$ and the multiplicities of $CC(j_!\cL)$ are smaller than $\fd(\lc_D(\cL)) \cdot \rk_{\Lambda} \cL$.

\end{thm}

\begin{proof}
Combine \cref{Bounddmdeg_normal} with \cref{bounded_ramification_ex}.
\end{proof}

\begin{lem}\label{points_avoinding_hypersurface_absolute}
Let $E$ be a vector space over $k$ and let $m\geq \binom{\rk E-1+d}{d}$.
Then, the set of points $(x_1,\dots, x_m) \in   \bP_k(E)^m$ such that no hypersurface of degree $d$ in $\bP_k(E)$ contains all the $x_i$ contains a dense open subset.
\end{lem}

\begin{proof}
We can suppose that $m=\binom{\rk E-1+d}{d}$.
Let $i_d : \bP_k(E) \hookrightarrow  \bP^{m-1}_k$ be the $d$-Veronese embedding and let $V_{d}\subset \bP^{m-1}_k$ be its image.
Let $U\subset (\bP^{m-1}_k)^m$ be the dense open set of $m$ ordered points in generic position in $\bP^{m-1}_k$. 
That is, any $(x_1,\dots, x_m) \in U$ spans $\bP^{m-1}_k$.
Then, for every $(x_1,\dots, x_m) \in   \bP_k(E)^m$, we have
\begin{align*}
& \text{ no hypersurface of degree $d$ in $\bP_k(E)$ contains all the $x_i$} \\
\Longleftrightarrow &  
\text{ no hyperplane in $\bP^{m-1}_k$ contains all the $i_d(x_i)$} \\
\Longleftrightarrow & \text{ $(i_d(x_1),\dots, i_d(x_m))\in U$}.
\end{align*}
Hence, we are left to show that $U\cap V_{d}^{m}$ is  dense open in $V_{d}^m$.
Since $V_{d}^m$ is irreducible and since $U$ is open, it is enough to show that  $U\cap V_{d}^{m}$ is not empty.
Since $V_{d}$ is not contained in any hyperplane of $\bP^{m-1}_k$, one can recursively construct $m$ points on $V_{d}$ spanning $\bP^{m-1}_k$  and the conclusion follows.
\end{proof}

\begin{thm}\label{micro_local_Lefschetz_hyperplane}
Let $(X,\Sigma)$ be a projective stratified scheme over $k$ algebraically closed.
Then, there is a closed immersion $i : X \hookrightarrow \bP_k(E)$, an admissible function $\fc : \bQ[\Coh(X)]\to \bQ$ and  $\fd\in \mathds{N}[x]$ such that for every  $r\geq 1$, every $\cE\in \bQ[\Coh(X)]$, there is a dense open subset $U\subset (\bP_k(E)^{\vee})^{N}$ with $N:=\binom{\rk E-1+\fd(\fc(\cE) ) \cdot r}{\fd(\fc(\cE) ) \cdot r}$ satisfying the following property : for every finite local ring $\Lambda$ of residue characteristic $\ell\neq p$, every $(H_1,\dots, H_N)\in U$ and every $\cP \in \Perv_{\Sigma,tf}^{\leq r}(X,
\cE,\Lambda)$, there is $1\leq j\leq N$ such that $H_j \hookrightarrow \bP_k(E)$ is $SS(i_\ast\cP)$-transversal.
\end{thm}

\begin{proof}
Let $\iota : X\hookrightarrow Y$ be a closed embedding in some smooth projective scheme over $k$.
Let $\Sigma_Y$ be the stratification on $Y$ given by $Y-X$ and the $Z\in \Sigma$.
Choose a closed immersion $i_Y : Y \hookrightarrow \bP_k(E)$ satisfying \cref{relative_beilinson_assumption} and assume that \cref{micro_local_Lefschetz_hyperplane} holds for $(Y,\Sigma_Y,i_Y)$. 
Let 
$\fc_Y : \bQ[\Coh(Y)]\to \bQ$ and $\fd_Y\in \mathds{N}[x]$ as given by \cref{micro_local_Lefschetz_hyperplane}.
If we put $i := i_Y \circ \iota :  X \hookrightarrow \bP_k(E)$, then $\fc_Y \circ \iota_* : \bQ[\Coh(X)]\to \bQ$ and $\fd_Y\in \mathds{N}[x]$ do the job.
Hence, we are left to prove \cref{micro_local_Lefschetz_hyperplane} in the case where $X$ is smooth over $k$ and  we have a closed immersion $i : X \hookrightarrow \bP_k(E)$ satisfying \cref{relative_beilinson_assumption}.\\ \indent
Apply \cref{Bounddmdeg_normal} to $(X,\Sigma)$ and $i : X \hookrightarrow \bP_k(E)$ with $a=-\dim X$ and $b=0$.
Let $\fc : \bQ[\Coh(X)]\to \bQ$ and  $\fd\in \mathds{N}[x]$ be the corresponding data and pick $r\geq 1$ and $\cE\in \bQ[\Coh(X)]$.
By \cref{points_avoinding_hypersurface_absolute}, there is a dense open subset $U\subset (\bP_k(E)^{\vee})^{N}$  of points $(H_1,\dots, H_N)$ such that no hypersurface of $\bP_k(E)^{\vee}$ of degree $\leq \fd(\fc(\cE) ) \cdot r$ contains all the $H_i$.
Let $\Lambda$ be a finite local ring of residue characteristic $\ell\neq p$, let $(H_1,\dots, H_N)\in U$ and $\cP \in \Perv_{\Sigma,tf}^{\leq r}(X,
\cE,\Lambda)$.
By \cref{bound_deg_SS}, we have
$$
\deg p^{\vee}(\bP(SS(i_{\ast}\cP))) \leq |\totdeg_{i}(CC(\cP)) |\leq 
\fd(\fc(\cE)) \cdot r \ .
$$
Hence, there is $1\leq j\leq N$ such that $H_j \notin  p^{\vee}(\bP(SS(i_{\ast}\cP)))$.
Thus $H_j \hookrightarrow \bP_k(E)$ is $SS(i_\ast\cP)$-transversal in virtue of \cref{transversality_H}.
\end{proof}

\section{Wild Lefschetz theorem for finite coefficients}


\begin{defin}\label{monodromy_group}
Let $U$ be a connected scheme of finite type over $k$ and let $\ubar\to U$ be a geometric point.
Let $\cL\in \Loc_{\ft}(U,\Lambda)$ and let $\rho : \pi_1(U,\ubar)\to \GL(\cL_{\ubar})$ be the corresponding representation.
We define the \textit{monodromy group of $\cL$ at $\ubar$} as the image of $\rho$.
\end{defin}

\begin{defin}\label{same_monodromy_group}
In the setting of \cref{monodromy_group}, let $f : V\to U$ be a morphism between connected schemes of finite type over  $k$.
Let $\overline{v}\to V$ be a geometric point and put $\ubar=f(\overline{v})$.
We say that $\cL$ and $f^*\cL$ \textit{have the same monodromy group} if the monodromy group of $\cL$ at $\ubar$ is the monodromy group of $f^*\cL$ at $\overline{v}$.
\end{defin}

\begin{rem}
Since two geometric points of $V$ are connected by a path, having the same monodromy group does not depend on the choice of a base point of $V$.
\end{rem}

The following lemma is elementary :

\begin{lem}\label{same_monodromy_group_ss_abstract}
Let $G_1\to G_2$ be a morphism of groups. 
Let $r\geq 0$, let $A$ be a commutative ring and let $\rho : G_2\to \GL_r(A)$ be a representation.
Assume that $\rho$ and $\rho|_{G_1}$ have the same image.
Then, if  $\rho$ is simple (resp. semi-simple), the representation $\rho|_{G_1}$ is simple (resp. semi-simple).
\end{lem}

\begin{cor}\label{same_monodromy_group_ss}
In the setting of \cref{same_monodromy_group}, assume that  $\cL$ and $f^*\cL$ have the same monodromy group.
Then, if  $\cL$ is simple (resp. semi-simple), the pullback $f^*\cL$ is simple (resp. semi-simple).
\end{cor}

\begin{construction}\label{representation_local_system}
Let $U$ be a connected scheme of finite type over $k$ and let $\ubar\to U$ be a geometric point.
Let $\cL\in \Loc_{\ft}(U,\Lambda)$ and let $\rho : \pi_1(U,\ubar)\to \GL(\cL_{\ubar})$ be the corresponding representation.
Let $G$ be the monodromy group of $\cL$ at $\ubar$.
We define $\Lambda[\cL,\ubar]$ as the object of $\Loc_{\ft}(U,\Lambda)$ corresponding to the representation 
$$
\pi_1(U,\ubar)\to G\to \GL(\Lambda[G])
$$
where the first arrow is induced by $\rho$ and where the second arrow is the regular representation of $G$.
\end{construction}

\begin{rem}\label{independence_in_u}
The isomorphism class of $\Lambda[\cL,\ubar]$ does not depend on $\ubar$.
\end{rem}

\begin{lem}\label{representation_local_system_lc_bound}
In the setting of \cref{representation_local_system}, let $j : U\to X$ be an open immersion, let $\cE\in \bQ[\Coh(X)]$ and assume that $j_!\cL$ has log conductors bounded by $\cE$.
Then, so does $j_!\Lambda[\cL,\ubar]$.
\end{lem}

\begin{proof}
Immediate from the fact that the action of $\pi_1(U,\overline{x})$ on the fibre of $\Lambda[\cL,\ubar]$ at a geometric point $\overline{x}\to U$ factors through that on $\cL_{\overline{x}}$.
\end{proof}

\begin{lem}\label{same_monodromy_criterion}
Let $G_1\to G_2$ be a morphism of groups. 
Let $r\geq 0$, let $A$ be a commutative ring and let $\rho : G_2\to \GL_r(A)$ be a representation with finite image $G$. 
Let $\phi : G\to \GL(A[G])$ be the regular representation of $G$, so that $G_1$ and $G_2$ act on $A[G]$.
Then $A[G]^{G_1}=A[G]^{G_2}$ if and only if the representations $\rho$ and $\rho|_{G_1}$ have the same image.
\end{lem}

\begin{proof}
The converse implication is obvious.
Let us prove the direct implication.
Let $H$ be the image of  $\rho|_{G_1}: G_1\to  \GL_r(A)$.  
We have $H\subset G$ and we have to show that $H=G$.
By assumption, we have
$$
\sum_{h\in H} h\in A[G]^{G_1}=A[G]^{G_2}=A\cdot  \sum_{g\in G} g \ .
$$
This is possible only if $H=G$.
\end{proof}

\begin{lem}\label{sheaf_criterion_same_monodromy}
Let $f : V\to U$ be a morphism between connected schemes of finite type over $k$.
Let $\ubar\to U$ be a geometric point.
Let $\cL\in \Loc_{tf}(U,\Lambda)$.
The following conditions are equivalent :
\begin{enumerate}\itemsep=0.2cm
\item The sheaves $\cL$ and $f^*\cL$ have the same monodromy group.
\item The canonical morphism
$$
H^0(U,\Lambda[\cL,\ubar]) \to H^0(V,f^*\Lambda[\cL,\ubar]) 
$$
is an isomorphism.
\end{enumerate}
\end{lem}

\begin{proof}
Immediate from \cref{same_monodromy_criterion}.
\end{proof}

The following proposition provides a criterion for \cref{sheaf_criterion_same_monodromy}-(2) to hold in terms of the singular support.

\begin{prop}\label{transversality_and_Hom}
Let $X$ be a projective scheme of finite type of pure dimension $n\geq 2$ over $k$ algebraically closed. 
Let $j : U\hookrightarrow X$ be an affine open immersion with $U$ smooth over $k$.
Let $i : X \hookrightarrow \mathds{P}_k$ be a closed immersion.
Let $H\subset \mathds{P}_k$ be a hyperplane such that $U\cap H$ is smooth of pure dimension $n-1$. 
Let $\cL,\cM \in \Loc_{\ft}(U,\Lambda)$ such that $H \hookrightarrow \mathds{P}_k$ is $SS(i_*j_!\cHom(\cL,\cM))$-transversal.
Then, the canonical morphism
$$
\Hom(\cL,\cM)\to \Hom(\cL|_{U\cap H},\cM|_{U\cap H})
$$
is an isomorphism.
\end{prop}

\begin{proof}
Consider the diagram with cartesian squares
$$
	\begin{tikzcd}
U\cap H\arrow{r}\arrow{d}{j'}&U\arrow{d}{j}  \\
X\cap H\arrow{r}  \arrow{d}{i'} &X \arrow{d}{i} \\
 H\arrow{r} & \mathds{P}_k    \ .
	\end{tikzcd}
$$
We have 
$$
\Hom(\cL,\cM) = H^0(U,\cHom(\cL,\cM))=H^0(\mathds{P}_k,i_* Rj_*\cHom(\cL,\cM))
$$
and 
$$
\Hom(\cL|_{U\cap H},\cM|_{U\cap H}) =H^0(H,i'_* Rj'_*\cHom(\cL|_{U\cap H},\cM|_{U\cap H})) \ .
$$
Note that $\cHom(\cL,\cM)[n]$ is perverse of finite tor-dimension on $U$.
\cref{automatic_transversality} gives
$$
   SS(i_*Rj_*\cHom(\cL,\cM))= SS(i_*j_!\cHom(\cL,\cM) ) \ .
$$
Hence, $H\hookrightarrow\mathds{P}_k$ is $SS(i_*Rj_*\cHom(\cL,\cM))$-transversal.
On the other hand, the map $U\cap H\hookrightarrow U$ is tautologically $SS(\cHom(\cL,\cM))$-transversal.
By assumption, we have
$$
\dim U\cap H -\dim U = -1 = \dim H - \dim \mathds{P}_k
$$
where all the above schemes are smooth of pure dimension.
Hence, \cref{basechange} implies that the base change 
$$
(i_* Rj_* \cHom(\cL,\cM) )|_H \to i'_* Rj'_* \cHom(\cL|_{U\cap H},\cM|_{U\cap H}) 
$$
is an isomorphism.
Put $\mathds{A}_k:=\mathds{P}_k-H$ and let $\jmath :  \mathds{A}_k \hookrightarrow  \mathds{P}_k$ is the inclusion.
Then, there is a distinguished triangle
$$
\jmath _! \jmath^* i_* Rj_* \cHom(\cL,\cM)  \to i_* Rj_* \cHom(\cL,\cM)  \to i'_* Rj'_* \cHom(\cL|_{U\cap H},\cM|_{U\cap H})  \ .
$$
Hence, to conclude the proof of \cref{transversality_and_Hom}, we are left to show that 
$$
H^1_c(\mathds{A}_k,\jmath^* i_* Rj_* \cHom(\cL,\cM) )
$$
vanishes.
Since $\jmath^* i_* Rj_* \cHom(\cL,\cM)[n]$ is  perverse on $\mathds{A}_k$ and since $\mathds{A}_k$ has dimension at least $2$, the sought-after vanishing thus follows from Artin's vanishing theorem for perverse sheaves \cite[Corollaire 4.1.2]{BBD}, \cite[XIV, Théorème 3.1]{SGA4-3}.
\end{proof}
 

\begin{thm}\label{Wild_Lefschetz_hyperplane}
Let $X$ be a projective  scheme of pure dimension $n\geq 2$ over $k$ algebraically closed.
Let $j : U\hookrightarrow X$ be an affine open immersion with $U$ smooth connected over $k$.
Then, there is a closed immersion $i : X\hookrightarrow \bP_k(E)$, there is an admissible  function $\fc : \bQ[\Coh(X)]\to \bQ$ and $\fd\in \mathds{N}[x]$ such that for every finite local ring  $\Lambda$ of residue characteristic $\ell \neq p$, every  $r\geq 1$, every $\cE\in \bQ[\Coh(X)]$,  
there is a dense open subset $V\subset (\bP_k(E)^{\vee})^{N}$ with 
$$
N:=\binom{\rk E-1+\fd(\fc(\cE) ) \cdot |\GL_r(\Lambda)|}{\fd(\fc(\cE) ) \cdot |\GL_r(\Lambda)|}
$$ 
satisfying the following  : for every $(H_1,\dots, H_N)\in V$ and every $\cL\in \Loc^{\leq r}_{tf}(U,\Lambda)$ where $j_!\cL$ has log conductors bounded by $\cE$, there is $1\leq a \leq N$ such that $\cL$ and $\cL|_{U\cap H_a}$ have the same monodromy group.
\end{thm}

\begin{proof}
Put $\Sigma=\{U,X-U\}$.
Let $i : X \hookrightarrow \bP_k(E)$,  $\fc : \bQ[\Coh(X)]\to \bQ$ and  $\fd\in \mathds{N}[x]$ as given by  \cref{micro_local_Lefschetz_hyperplane} applied to $(X,\Sigma)$.
Let  $r\geq 1$ and  $\cE\in \bQ[\Coh(X)]$.
Let  $V\subset (\bP_k(E)^{\vee})^{N}$ with $N:=\binom{\rk E-1+\fd(\fc(\cE) ) \cdot |\GL_r(\Lambda)|}{\fd(\fc(\cE) ) \cdot |\GL_r(\Lambda)|}$ be the dense open subset attached to $(|\GL_r(\Lambda)|,\cE)$ by  \cref{micro_local_Lefschetz_hyperplane}.
At the cost of shrinking $V$, we can assume that for every  $(H_1,\dots, H_N)\in V$ and every $1\leq a\leq N$,  the hyperplane $H_a$ is transverse to $U$ and $U\cap H_a$ is connected.
Let us prove that $V$ meets our requirements.
\\ \indent
Let $(H_1,\dots, H_N)\in V$ and $\cL\in \Loc^{\leq r}_{tf}(U,\Lambda)$ where $j_!\cL$ has log conductors bounded by $\cE$.
Let $\ubar\to U$ be a geometric point.
Then \cref{representation_local_system_lc_bound} implies that $j_!\Lambda[\cL,\ubar]$ has log conductors bounded by $\cE$ as well.
Since $j : U\hookrightarrow X$ is affine, we deduce
$$
j_!\Lambda[\cL,\ubar][n]\in \Perv_{\Sigma, tf}(X,\cE,\Lambda) \ .
$$
Since $j_!\Lambda[\cL,\ubar][n]$  has rank smaller than $|\GL_r(\Lambda)|$
there is $1\leq a \leq N$ such that $H_a \hookrightarrow \bP_k(E)$ is $SS(i_\ast j_!\Lambda[\cL,\ubar])$-transversal.
By \cref{transversality_and_Hom}, we deduce that the canonical morphism
$$
H^0(U,\Lambda[\cL,\ubar]) \to H^0(U\cap H_a,\Lambda[\cL,\ubar]|_{U\cap H_a}) 
$$
is an isomorphism.
By \cref{sheaf_criterion_same_monodromy}, the sheaves  $\cL$ and $\cL|_{U\cap H_a}$ have the same monodromy group.
\end{proof}

\textit{Proof of  \cref{thm_5}}. 
     We are going to use \cref{Wild_Lefschetz_hyperplane}.
     Since $D:=X-U$ is the support of an effective Cartier divisor, the open immersion $j : U\hookrightarrow X$ is affine.
    Let then $i,\mu,P$ as given by \cref{Wild_Lefschetz_hyperplane} and fix a finite local ring $\Lambda$ of residue characteristic $\ell \neq p$.
    Let  $D_1,\dots, D_a$ be the irreducible components of $D$, let $R=m_1 D_1+\dots +m_a D_a$ be an effective Cartier divisor supported on $D$ and fix $r\geq 0$.
     We need to define an integer $N(\deg R,r)$ as in \cref{thm_5}.
     Saying that $\cL\in \Loc_{tf}(U,\Lambda)$ has log conductors bounded by $R$ in the sense of \cref{def_provisoire} means that $j_!\cL$ has log conductors bounded by $\cO_R$ in the sense of \cref{bounded_conductor}. 
     By \cref{Wild_Lefschetz_hyperplane} applied to $\cE=\cO_R$, if we put 
     $$
     M(R,r,\Lambda) := \binom{\dim \bP_k+\fd(\fc(\cO_R) ) \cdot |\GL_r(\Lambda)|}{\fd(\fc(\cO_R) ) \cdot |\GL_r(\Lambda)|} \ ,
     $$
     there is a dense open subset 
     $V\subset (\bP_k^{\vee})^{M(R,r,\Lambda)}$ 
satisfying the following  : for every $(H_1,\dots, H_{M(R,r,\Lambda)})\in V$ and every $\cL\in \Loc^{\leq r}_{tf}(U,\Lambda)$ where $\cL$ has log conductors bounded by $R$, there is $1\leq a \leq M(R,r,\Lambda)$ such that $\cL$ and $\cL|_{U\cap H_a}$ have the same monodromy group.
     On the other hand, we have 
\[
\mu(\cO_R)=\mu(\bigoplus_{i=1}^a m_i \cdot \cO_{D_i})\leq \deg R\cdot \max_i \{\mu(\cO_{D_i})\} \ .
\]
If we define $N : \bN^2 \to \bN$ by 
\[
(m,r) \to \binom{\dim \bP_k+\fd(m\cdot \max_i \{\mu(\cO_{D_i})\} ) \cdot |\GL_r(\Lambda)|}{\fd(m\cdot \max_i \{\mu(\cO_{D_i})\} ) \cdot |\GL_r(\Lambda)|} \ ,
\]
then $M(R,r,\Lambda)\leq N(\deg R,r)$ and the conclusion follows.

\section{Moduli of multi-flags}\label{multi_flag}

\begin{notation}
For a scheme $S$, we will denote by $\Sch_S$ the category of schemes over $S$.
\end{notation}

\begin{recollection}
Let $S$ be a scheme of finite type over $k$ and let $E$ be a locally free sheaf of $\cO_S$-modules on $S$.
For integers $1\leq n_1<\dots < n_r \leq \rk E$, we denote by  $\Fl_S(E,n_1,\dots, n_r)$ the scheme over $S$ of $(n_1,\dots, n_r)$-flags in $E$.
Recall that $\Fl_S(E,n_1,\dots, n_r)$ is smooth over $S$.
By definition, $\Fl_S(E,n_1,\dots, n_r)$ represents the functor sending $f : T\to S$ to the set of isomorphism classes of epimorphisms 
$$
f^*E \to E_r \to \cdots \to E_1
$$
where  $E_i$ is a locally free sheaf of rank $n_i$ on $T$ for every $1\leq i \leq r$.
If $r=1$, we note $\bG_S(E,n)$ for $\Fl_S(E,n)$ and if $r=n=1$, we note $\bP_S(E)$ for $\bG_S(E,1)$.
\end{recollection}

\begin{construction}\label{hypersurface_vanishing_at_point}
Let $S$ be a scheme of finite type over $k$ and let $E$ be a locally free sheaf of $\cO_S$-modules on $S$.
Let $\pi : \bP_S(E) \to S$ be the structural morphism.
For $d\geq 1$, the universal locally free rank one quotient
$$
\pi^*E \to \cO_{\bP_S(E) }(1) 
$$
induces a short exact sequence of locally free sheaves
\begin{equation}\label{degree_d_sequence}
0 \to K \to \pi^* \Sym^d E \to \cO_{\bP_S(E) }(d) \to 0
\end{equation}
Intuitively, the fiber $K_x$ of $K$ at a point $x\in \bP_S(E)$  over $s\in S$ is the hyperplane of homogeneous polynomials of degree $d$ of $\bP_s(E_s)$  vanishing at $x$. 

\end{construction}


\begin{lem}\label{points_avoinding_hypersurface}
In the setting of \cref{hypersurface_vanishing_at_point}, for every $m\geq \binom{\rk E-1+d}{d}$, there is a dense open subset 
$$
U\subset \bigtimes_{S}^m \bP_S(E)
$$ 
of the  $m$-th self product of $\bP_S(E)$ in $\Sch_S$ such that for every $(x_1,\dots, x_m) \in U(\sbar)$ over a geometric algebraic point $\sbar \to S$, no hypersurface of degree at most $d$ in $\bP_{\sbar }(E_{\sbar })$ contains all the $x_i$.
\end{lem}

\begin{proof}
It is enough to find a dense open subset $U\subset \bigtimes_{S}^m \bP_S(E)$ such that for every  $(x_1,\dots, x_m) \in U(\sbar )$  over a geometric algebraic point $\sbar \to S$, no hypersurface of degree exactly $d$ in $\bP_{\sbar }(E_{\sbar })$ contains all the $x_i$.
For $1\leq i \leq m$, let 
$$
\pi_i : \bigtimes_{S}^m \bP_S(E) \to \bP_S(E) 
$$
be the projection on the $i$-th factor and let 
$$
p : \bigtimes_{S}^m \bP_S(E) \to S
$$
be the structural morphism.
Consider the exact sequence of coherent sheaves 
$$
0 \to \bigcap_{i=1}^m  \pi_i^{\ast} K \to p^* \Sym^d E \to \bigoplus_{i=1}^m \pi_i^{\ast} \cO_{\bP_S(E) }(d)  
$$
on $\bigtimes_{S}^m \bP_S(E)$.
Let $F$ be the image of the last arrow and consider the short exact sequence
$$
0 \to F \to \bigoplus_{i=1}^m \pi_i^{\ast} \cO_{\bP_S(E) }(d)  \to G \to 0 \ .
$$
Choose a dense open subset $V\subset \bigtimes_{S}^m \bP_S(E)$ such that $F$ and $G$ are locally free on $V$.
Then, for every  $x=(x_1,\dots,x_m)\in V(\sbar )$  over a geometric algebraic point $\sbar \to S$, there is a canonical isomorphism 
$$
x^* \bigcap_{i=1}^m  \pi_i^{\ast} K \simeq \bigcap_{i=1}^m  x_i^* K 
$$
where the right-hand side is the subspace of $\Sym^d E_{\sbar }$ of homogeneous polynomials of degree $d$ vanishing at $x_1,\dots, x_m$.
To prove \cref{points_avoinding_hypersurface}, we are thus left to show the existence of a dense open subset $U \subset V$ on which $\bigcap_{i=1}^m  \pi_i^{\ast} K$ vanishes.
Since the vanishing locus of a coherent sheaf is an open subset, we are left to show the existence of a dense set of points in $V$ on which $\bigcap_{i=1}^m  \pi_i^{\ast} K$ vanishes.
This follows from \cref{points_avoinding_hypersurface_absolute}.
\end{proof}

\begin{defin}\label{defin_tree}
A \textit{rooted tree} is a finite poset $\Tr$ with an initial object $0$ such that for every $v\in \Tr$, the subset  $\Tr^{\leq v} :=\{w\in \Tr \text{ with } w\leq v\}$ is totally ordered.
For a vertex $v\in \Tr,$ the natural number $d(v):=|\Tr^{\leq v}|-1$ is the  \textit{depth} of $v$.
The \textit{depth} of $\Tr$ is the maximal depth of its vertices.
A  \textit{branch} of $\Tr$ is a maximal totally ordered subset of $\Tr$.
A \textit{subtree} of $\Tr$ is a subposet $\Tr'\subset \Tr$ which is closed downwards.


\end{defin}

\begin{construction}\label{construction_multiflag}
Let $S$ be a scheme of finite type over $k$ and let $E$ be a non zero locally free sheaf of $\cO_S$-modules on $S$.
Let $\Tr$ be a rooted tree and let $W : \Tr^{\op} \to \mathds{N}^{\ast}$ be a strictly decreasing morphism of posets sending 0 to $\rk E$.
Define
$$
\Fl_S(E,\Tr,W) : \Sch^{\op}_S  \to \Ens
$$
as the functor sending $f : T \to S$ to the set of isomorphism classes of diagrams 
$$
E_{\bullet} : \Tr\to \Coh(T)
$$ 
where 
\begin{enumerate}\itemsep=0.2cm
\item $E_0 = f^*E$ and for every $v\in \Tr$, the sheaf $E_v$ is  locally free of rank $W(v)$.

\item The arrows of $\Tr$ are sent to epimorphisms of $\Coh(T)$.
\end{enumerate}

\end{construction}

\begin{rem}\label{subspace_viewpoint}
The diagram $E_{\bullet} : \Tr\to \Coh(T)$ is equivalent to an ordered set of projective subspaces over $T$
$$
F_v \subset \bP_T(f^*E),  \    v\in \Tr
$$
where $F_v$ has relative dimension $W(v)-1$ over $T$.
We will tacitly use both descriptions.
\end{rem}

\begin{rem}
If $W=\rk E-d(-)$ where $d(-)$ is the depth function, we simply note $\Fl_S(E,\Tr)$  for $\Fl_S(E,\Tr,W)$.
\end{rem}

\begin{example}\label{multiflag_n}
If $\Tr$  has a single branch, then $\Fl_S(E,\Tr)$  is a partial flag functor.
In particular, $\Fl_S(E,\Tr)$ is a smooth scheme of finite type over $S$.
\end{example}

\begin{example}\label{multiflag_1}
If every element of $\Tr$ distinct from $0$ has depth $1$, then
$$
\Fl_S(E,\Tr,W)=\prod_{v\in \Tr\setminus \{0\}}\mathds{G}_S(E,W(v)) 
$$
where the product is performed in the category of schemes over $S$.
\end{example}

\begin{lem}\label{flat_tree_is_a_scheme}
The functor $\Fl_S(E,\Tr,W)$ is a scheme of finite type over $S$.
\end{lem}
\begin{proof}
We argue by recursion on the depth $d$ of $\Tr$.
The case of depth $0$ follows from \cref{multiflag_n}.
Assume that $\Tr$ has depth $d>0$.
Let $\Tr^{\leq d-1} \subset \Tr$ be the subtree of vertices of depth smaller than $d-1$ and $\Tr^{ d} \subset \Tr$ be the set of vertices of depth $d$.
Then, the square
$$
  \begin{tikzcd}
  \bigsqcup_{v\in  \Tr^{d} }  \Tr^{< v}  \arrow{r} \arrow{d}   &  \Tr^{\leq d-1}\arrow{d}    \\
\bigsqcup_{v\in  \Tr^{d} }  \Tr^{\leq v}\arrow{r}                   & \Tr
  \end{tikzcd}
$$
is a pushout in the category of posets.
Thus, the induced square 
$$
  \begin{tikzcd}
  \Fl_S(E,\Tr,W)  \arrow{r} \arrow{d}   &  \Fl_S(E,\Tr^{\leq  d-1},W|_{\Tr^{\leq d-1}})\arrow{d}    \\
\prod_{v\in \Tr^{d}  }  \Fl_S(E,\Tr^{\leq  v},W|_{\Tr^{\leq v}}) \arrow{r}                   & \prod_{v\in \Tr^{d}  }   \Fl_S(E,\Tr^{< v},W|_{\Tr^{< v}}) 
  \end{tikzcd}
$$
in $\Fun(\Sch^{\op}_S,\Ens)$ is cartesian.
By recursion assumption, the upper right functor is a scheme of finite type over $S$.
By \cref{multiflag_n}, so are the bottom functors.
The conclusion thus follows.
\end{proof}

\begin{lem}\label{Fl_and_Gr}
In the setting of \cref{construction_multiflag}, let $v,w\in \Tr$ where  $w$ is an immediate successor of $v$ and is  maximal in $\Tr$.
Put $\Tr^{\circ} := \Tr \setminus \{w\}$ and $W^{\circ}:=W|_{\Tr \setminus \{w\}}$. 
Let 
$$
\cE_{\bullet} : \Tr^{\circ} \to \Coh(\Fl_S(E,\Tr^{\circ},W^{\circ}))
$$ 
be the universal object of $\Fl_S(E,\Tr^{\circ},W^{\circ})$.
Then the restriction
$$
\Fl_S(E,\Tr,W)\to \Fl_S(E,\Tr^{\circ},W^{\circ})
$$
exhibits $\Fl_S(E,\Tr,W)$ as the Grassmannian of $W(w)$-plans in $\cE_{v}$ over $\Fl_S(E,\Tr^{\circ} ,W^{\circ})$.
\end{lem}

\begin{proof}
For $f : T\to \Fl_S(E,\Tr^{\circ},W^{\circ})$, we need to construct a bijection
$$
\alpha_T : \Gr_{\Fl_S(E,\Tr^{\circ},W^{\circ})}(\cE_{v},W(w))(T)\to \Fl_S(E,\Tr,W)(T)
$$
natural in $T$.
An element of the left hand side is the datum of an epimorphism
\begin{equation}\label{eqFl_and_Gr}
h : f^*\cE_{v}  \to F
\end{equation}
where $F$ is a locally free sheaf of rank $W(w)$ on $T$.
Since $w$ is maximal in  $\Tr$, concatenating the pullback diagram
$$
f^* \cE_{\bullet} : \Tr^{\circ} \to \Coh(T)
$$
with \eqref{eqFl_and_Gr} gives rise to a diagram
$$
\alpha_T(h) : \Tr \to \Coh(T)
$$
natural in the choice of $f : T\to \Fl_S(E,\Tr^{\circ},W^{\circ})$.
One readily checks that $\alpha$ is an isomorphism of functors.
\end{proof}

\begin{lem}\label{restriction_subtree_smooth}
In the setting of \cref{construction_multiflag}, let $\Tr' \subset \Tr$ be a subtree  and put $W':=W|_{\Tr' }$.
Then, the induced morphism over $S$
$$
\Fl_S(E,\Tr,W)\to \Fl_S(E,\Tr',W')
$$
is smooth projective surjective of relative pure dimension 
$$
\sum_{v\in \Tr} \sum_{\substack{w\in  \Tr\setminus \Tr' ,v\leq w  \\ d(w)=d(v)+1}}  W(w)(W(v)-W(w)) \ .
$$
In particular $\Fl_S(E,\Tr,W)$ is smooth proper surjective over $S$ of relative pure dimension 
$$
\sum_{\substack{v\leq w \\ d(w)=d(v)+1}}  W(w)(W(v)-W(w)) \ .
$$
\end{lem}

\begin{proof}
Since $\Tr'$ is obtained from $\Tr$ by  successive deletions of maximal vertices, we can suppose that  $\Tr' =\Tr \setminus \{w\}$ where $w\in \Tr$ is maximal.
Then \cref{restriction_subtree_smooth} follows from \cref{Fl_and_Gr}.
\end{proof}

\begin{lem}\label{restriction_subtree_smooth_irr}
In the setting of \cref{construction_multiflag},  assume that $S$ is irreducible.
Then so is $\Fl_S(E,\Tr,W)$.
\end{lem}
\begin{proof}
Recall that by going-down property, for every flat morphism $f : X\to S$ over an irreducible base and with irreducible generic fiber, the source $X$ is irreducible.
To prove the irreducibility of $\Fl_S(E,\Tr,W)$, we argue recursively on the cardinality of $\Tr$.
If $\Tr=\{0\}$, we have $\Fl_S(E,\Tr,W)\simeq S$ and there is nothing to prove.
Let $w\in \Tr$ be a maximal element.
Put $\Tr^{\circ} := \Tr \setminus \{w\}$ and $W^{\circ}:=W|_{\Tr \setminus \{w\}}$. 
Then, \cref{Fl_and_Gr} ensures that  $\Fl_S(E,\Tr,W)$ identifies canonically with a Grassmannian over $\Fl_S(E,\Tr^{\circ},W^{\circ})$.
By recursion assumption, $\Fl_S(E,\Tr^{\circ},W^{\circ})$ is irreducible.
We thus conclude with the above observation and the fact that the grassmannian over a field is irreducible.
\end{proof}

\begin{lem}\label{fiber_product_grass}
In the setting of \cref{construction_multiflag}, let  $v\in \Tr$ such that every immediate successor $w_1,\dots, w_n$ of $v$ is maximal.
Let $\Tr'\subset \Tr$ be the subtree obtained from $\Tr$ by removing the immediate successors of $v$ and put $W':=W|_{\Tr'}$. 
Let 
$$
\cE_{\bullet} : \Tr' \to \Coh(\Fl_S(E,\Tr',W'))
$$ 
be the universal object of $\Fl_S(E,\Tr',W')$.
Then, the induced morphism over $S$
$$
\Fl_S(E,\Tr,W)\to \Fl_S(E,\Tr',W')
$$
exhibits $\Fl_S(E,\Tr,W)$ as the fiber product over $\Fl_S(E,\Tr',W')$ of the Grassmannian of $W(w_i)$-plans in $\cE_{v}$ over $\Fl_S(E,\Tr',W')$ where $1\leq i \leq n$.
\end{lem}

\begin{proof}
For $1\leq i \leq n$, put $\Tr'_i = \Tr' \cup \{w_i\}$ and $W_i :=W|_{\Tr'_i }$.
Then, the pushout square of posets 
$$
  \begin{tikzcd}
\bigsqcup_{i=1}^n    \Tr'         \arrow{r} \arrow{d}   &  \Tr'   \arrow{d}    \\
\bigsqcup_{i=1}^n   \Tr'_i   \arrow{r}                   & \Tr
  \end{tikzcd}
$$
gives rise to a pullback of schemes 
$$
  \begin{tikzcd}
\Fl_S(E,\Tr,W)          \arrow{r} \arrow{d}   &  \Fl_S(E,\Tr',W')     \arrow{d}    \\
\prod_{i=1}^n  \Fl_S(E,\Tr'_i,W_i)     \arrow{r}                   & \prod_{i=1}^n \Fl_S(E,\Tr',W')
  \end{tikzcd}
$$
where the right vertical arrow is the diagonal map.
Hence, there is a canonical isomorphism
$$
\Fl_S(E,\Tr,W)  \simeq \Fl_S(E,\Tr'_1,W_1)\times_{\Fl_S(E,\Tr',W')} \cdots \times_{\Fl_S(E,\Tr',W')}    \Fl_S(E,\Tr'_n,W_n)
$$
Thus, the conclusion follows from \cref{Fl_and_Gr}.
\end{proof}

\begin{defin}\label{ramified_tree}
Let $r,d\geq 1$.
We say that a rooted tree $\Tr$  \textit{ramifies enough with respect to $(r,d)$} if $\Tr$ has depth at most $r-2$, has at least two vertices and if every non maximal vertex $v\in \Tr$ has at least 
$$
\binom{r-1-d(v)+d}{d}
$$ 
immediate successors, where $d(-) : \Tr \to \mathds{N}$ is the depth function. 
\end{defin}

\begin{rem}
The above definition would make sense for $d$ rational number by replacing $d$ by $\lfloor d \rfloor$ in the above formula.
\end{rem}

\begin{rem}\label{smallest_ramified_tree}
The smallest rooted tree  $\Tr$ whose maximal vertices have all depth $\delta\leq r-2$ and ramifying enough with respect to $(r,d)$ has exactly
$$
N(r,d,a):=\prod_{i=0}^{a-1} \binom{r-1-i+d}{d}
$$
vertices of depth $a$ for every $0\leq a \leq \delta$.
In that case, \cref{restriction_subtree_smooth} gives that for every field $k$ and every finite dimensional vector space $E$ over $k$, we have
$$
\dim \Fl_k(E,\Tr) = \sum_{a=1}^{\delta} (\rk E-a)\cdot \prod_{i=0}^{a-1} \binom{r-1-i+d}{d} \ .
$$
We denote by $C(r,d,\delta)$ the above dimension.
\end{rem}

\begin{lem}\label{multiple_avoidance_hypersurface}
Let $S$ be a scheme of finite type over $k$, let $E$ be a locally free sheaf of $\cO_S$-modules on $S$ and let $d\geq 1$.
Let $\Tr$ be a rooted  tree ramifying enough with respect to $(\rk E,d)$.
Then, there is a dense open subset $U \subset \Fl_S(E,\Tr)$ such that for every geometric algebraic point $\sbar \to S$, every $F_{\bullet}\in U(\sbar)$ and every non maximal vertex $v\in \Tr$, no hypersurface of $F_v^{\vee}$ of degree at most $d$  contains all the $F_w$,  for $w$ immediate successor of $v$.
\end{lem}

\begin{proof}
Let $v\in \Tr$ be a non maximal vertex.
Since the intersection of a finite number of dense open subsets is again dense open, it is enough to find a dense open subset $U_v \subset \Fl_S(E,\Tr)$ such that for every geometric algebraic point $\sbar \to S$ and every $F_{\bullet}\in U_v(\sbar)$, no hypersurface of $F_v^{\vee}$ of degree at most $d$  contains all the $F_w$,  for $w$ immediate successor of $v$.
Let $\Tr'\subset \Tr$ be the subtree obtained by removing the vertices of depth at least $d(v)+2$.
By \cref{restriction_subtree_smooth}, the restriction morphism
$$
\Fl_S(E,\Tr)\to \Fl_S(E,\Tr')
$$
is flat.
By going-down property,  the pre-image of any dense open subset of $\Fl_S(E,\Tr')$ is dense open in $\Fl_S(E,\Tr)$.
Hence, at the cost of replacing $\Tr$ by $\Tr'$, we can suppose that every immediate successor of $v$ is maximal in $\Tr$.
In that case,  let $\Tr'\subset \Tr$ be the subtree obtained from $\Tr$ by removing the immediate successors of $v$.
Let 
$$
\cE_{\bullet} : \Tr' \to \Coh(\Fl_S(E,\Tr'))
$$ 
be the universal object of $\Fl_S(E,\Tr')$.
Then \cref{fiber_product_grass} gives a canonical isomorphism
$$
\Fl_S(E,\Tr)\simeq  \bG_{\Fl_S(E,\Tr')}(\cE_v, \rk \cE_v-1) \times_{\Fl_S(E,\Tr')} \cdots  \times_{\Fl_S(E,\Tr')}  \bG_{\Fl_S(E,\Tr')}(\cE_v, \rk \cE_v-1)
$$
where each factor corresponds to an immediate successor of $v$. 
Thus, 
$$
\Fl_S(E,\Tr)\simeq  \bP_{\Fl_S(E,\Tr')}(\cE_v^\vee) \times_{\Fl_S(E,\Tr')} \cdots  \times_{\Fl_S(E,\Tr')}  \bP_{\Fl_S(E,\Tr')}(\cE_v^\vee) \ .
$$
Since $\Tr$ ramifies enough, $v$ has at least 
$$
\binom{\rk E-1-d(v)+d}{d} =\binom{\rk \cE_v-1+d}{d}
$$
successors.
Thus, the existence of $U_v$  follows from \cref{points_avoinding_hypersurface}.
\end{proof}

The following lemma gives some flexibility in the choice of the dense open subset $U$ from \cref{multiple_avoidance_hypersurface}.

\begin{lem}\label{dense_invers_image}
In the setting of \cref{construction_multiflag},
let
\begin{equation}\label{eq_dense_invers_image}
\Fl_S(E,\Tr,W) \to \prod_{v\in \Tr\setminus \{0\}} \bG_S(E,W(v)) 
\end{equation}
be the morphism of schemes over $S$ induced by the injections of posets $\{0,v\}\subset \Tr$ for $v\in \Tr\setminus \{0\}$.
For every $1\leq a \leq \rk E$, choose a dense open subset $U_a\subset \bG_S(E,a)$.
Then, the inverse image of 
\begin{equation}\label{eq_dense_invers_image_open}
 \prod_{v\in \Tr\setminus \{0\}} U_{W(v)} 
\end{equation}
by \eqref{eq_dense_invers_image} is dense open in $\Fl_S(E,\Tr,W)$.
\end{lem}
\begin{proof}
Since  $\Fl_S(E,\Tr,W)$ is flat over $S$ by \cref{restriction_subtree_smooth}, the going-down property implies that the inverse image of a dense open subset of $S$ is dense in $\Fl_S(E,\Tr,S)$.
Thus, we can always replace $S$ by a dense open subset.
Hence, we can suppose that $S$ is irreducible.
By \cref{restriction_subtree_smooth_irr}, the scheme $\Fl_S(E,\Tr,W)$ is then irreducible.
Hence, to prove \cref{dense_invers_image}, we are left to show that the inverse image of 
\eqref{eq_dense_invers_image_open} by \eqref{eq_dense_invers_image} is not empty.
To do this, we argue by recursion on the cardinality of $\Tr$.
If $|\Tr|=1$, both \eqref{eq_dense_invers_image_open} and the target of \eqref{eq_dense_invers_image} are the terminal object of $\Sch_S$, that is $S$. 
In that case, there is thus  nothing to prove.
Assume that $|\Tr|>1$.
Choose a maximal element $w\in \Tr$ and let $v$ be the its immediate antecedent.
Put  $\Tr':=\Tr\setminus \{w\}$ and $W':=W|_{\Tr\setminus \{w\}}$ and consider the span
$$
  \begin{tikzcd}
\bG_S(E,W(v))   &   \Fl_S(E,W(w),W(v))   \arrow{l}{p}   \arrow{r}{q} & \bG_S(E,W(w))  \ .
  \end{tikzcd}
$$
By \cref{restriction_subtree_smooth}, the map $p$ is flat.
Thus,  $p(q^{-1}(U_{W(w)}))$ is an  open subset of $\bG_S(E,W(v))$.
Since $S$ is irreducible,  so is $\bG_S(E,W(v))$.
Hence,  $p(q^{-1}(U_{W(w)}))$ is a dense open subset of $\bG_S(E,W(v))$.
For $a\neq W(v)$,  put $U'_a=U_a$ and put
$$
U'_{W(v)} := U_{W(v)} \cap p(q^{-1}(U_{W(w)}))\ .
$$
By recursion assumption applied to $\Tr'$ and to the $U'_a$,  there is $F_{\bullet}\in \Fl_S(E,\Tr',W')$ such that $F_{v'}\in U'_{W(v')} \subset U_{W(v')}$ for every $v'\in \Tr'= \Tr\setminus \{w\}$.
By construction, $F_v\in  U_{W(v)}$ contains a $W(w)$-dimensional subspace  $F_w \in  U_{W(w)}$.
Then, completing $F_{\bullet}$ with $F_w$ gives a point of $\Fl_S(E,\Tr,W)$ meeting our requirement.
\end{proof}

\section{Higher codimension Lefschetz theorems via multi-flags}

The goal of this section is to extend \cref{micro_local_Lefschetz_hyperplane}
and \cref{Wild_Lefschetz_hyperplane} to higher codimension.

\subsection{Boundedness for total degrees}

\begin{construction}\label{family_projective_subspace}
We are going to apply \cref{Bounddmdeg_normal} to the families of sections by projective subspaces.
Let $X$ be a smooth projective scheme over $k$.
Let $i : X\hookrightarrow \bP_k(E)$ be a closed immersion satisfying \cref{relative_beilinson_assumption}.
For $2\leq r \leq \rk E$, let $\Fl_k(E,1,r)$
be the universal dimension $r-1$-projective subspace in $\bP_k(E)$  and
 consider the commutative diagram
$$
\begin{tikzcd}  
X(E,r)\arrow{r}{p_{r,X}}   \arrow{d}{i_r}  &  X\arrow{d}{i} \\
\Fl_k(E,1,r)\arrow{r} \arrow{d}  &   \bP_k(E)   \arrow{d}  \\
 \bG_k(E,r)  \arrow[leftarrow, uu,bend left = 70, near end,  "p_{r,X}^{\vee}"] \arrow{r}   &   \Spec k
\end{tikzcd}
$$
 with cartesian upper square.
\end{construction}

\begin{lem}\label{existence_Ua}
In the setting of \cref{family_projective_subspace}, there is a dense open subset $U \subset \bG_k(E,r)$ above which $p_{r,X}^{\vee} : X(E,r) \to  \bG_k(E,r) $ is smooth.
\end{lem}

\begin{proof}
By Bertini's theorem \cite[Théorème 6.10.2]{Jouanolou}, the generic fibre of $p_{r,X}^{\vee} : X(E,r) \to  \bG_k(E,r) $ is smooth.  
By generic flatness theorem and openness of the smooth locus \cite[Théorème 12.2.4]{EGAIV} for proper flat morphisms of finite presentation, there is a dense open subset $U \subset \bG_k(E,r)$ above which $p_{r,X}^{\vee} : X(E,r) \to  \bG_k(E,r) $ is flat with smooth fibres.
The conclusion follows.
\end{proof}

\begin{lem}\label{beilinson_assumption_on_subspace}
Let $X$ be a smooth projective variety over  $k$ and let  $i : X\hookrightarrow \bP_k(E)$ be a closed immersion satisfying \cref{relative_beilinson_assumption}.
Let $F\subset \bP_k(E)$ be a projective subspace such that 
$X\cap F$ is smooth.
Then,  $X\cap F\hookrightarrow F$ satisfies \cref{relative_beilinson_assumption}.
\end{lem}

\begin{proof}
Put  $\cL:=i^*\cO_{\bP_k(E)}(1)$ and let $u,v \in X\cap F$ be closed points.
By assumption, the top horizontal arrow of the following commutative diagram
$$
\begin{tikzcd}  
\Gamma(\bP_k(E),\cO_{\bP_k(E)}(1)) \arrow{r} \arrow{d}   &  \Gamma(X,\cL) \arrow{r} \arrow{d}  &  \cL_u/\mathfrak{m}^2_u \cL_u \oplus \cL_v/\mathfrak{m}^2_v \cL_v    \arrow{d}   \\
\Gamma(F,\cO_{F}(1)) \arrow{r}    &  \Gamma(X\cap F,\cL|_F) \arrow{r}  &  (\cL|_F)_u/\mathfrak{m}^2_u (\cL|_F)_u \oplus (\cL|_F)_v/\mathfrak{m}^2_v (\cL|_F)_v   
\end{tikzcd}
$$
is surjective.
Since the right vertical arrow is surjective as well, so is the bottom horizontal arrow.
\end{proof}

\begin{lem}\label{Ua_Beilinson_assumption}
In the setting of \cref{family_projective_subspace}, let
 $U \subset \bG_k(E,r)$ be an open subset above which $p_{r,X}^{\vee} : X(E,r) \to  \bG_k(E,r) $  is smooth.
Let $E' \in \Coh(\bG_k(E,r))$ be the universal locally free quotient of rank $r$ of $E$.
Then, via the  canonical identification 
$$
\Fl_k(E,1,r) \simeq\bP_{\bG_k(E,r)}(E')
$$
supplied by \cref{Fl_and_Gr}, the closed immersion
$$
i_r : X(E,r) \hookrightarrow \Fl_k(E,1,r) 
$$
over $\bG_k(E,r)$ satisfies \cref{relative_beilinson_assumption} above $U$.
\end{lem}

\begin{proof}
Let $\overline{x} \to U$ be an algebraic geometric point lying above an algebraic geometric point $\sbar\to \Spec k$ and let $F\subset \bP_{\sbar}(E_{\sbar})$ be the corresponding projective subspace.
Then, the pullback of $i_r$ over $\overline{x}$  reads as $X_{\sbar}\cap F \hookrightarrow F$.
Thus, \cref{Ua_Beilinson_assumption} follows from \cref{beilinson_assumption_on_subspace}.
\end{proof}

\begin{lem}\label{bounded_degree_on_subspace}
Let $(X,\Sigma)$ be a smooth projective stratified scheme over $k$ algebraically closed.
Let $i : X\hookrightarrow \bP_k(E)$ be a closed immersion  satisfying \cref{relative_beilinson_assumption}.
Let $2 \leq r \leq \rk E$ and $a\leq b$ be integers.
Then, there is  an admissible function $\fc :\bQ[\Coh(X)]\to \bQ$ and $\fd\in \mathds{N}[x]$, there is a dense open subset $U\subset \bG_k(E,r)$ such that for every finite local ring $\Lambda$ of residue  characteristic $\ell \neq p$, every  $F \in U(k)$, every $\cE\in \bQ[\Coh(X)]$  and $\cK\in D_{\Sigma|_{X\cap F},\ft}^{[a,b]}(X\cap F,\cE|_{X\cap F},\Lambda)$,  we have :
\begin{enumerate}\itemsep=0.2cm
\item  The scheme $X\cap F$  is smooth.

\item  $|\totdeg_{i_F}(CC(\cK)) | \leq \fd(\fc(\cE)) \cdot \Rk_{\Lambda} \cK$ where $i_F : X\cap F \hookrightarrow F$ is the inclusion.
\end{enumerate}

\end{lem}
\begin{proof}
We use the notations from \cref{family_projective_subspace}.
By \cref{existence_Ua}, there is a dense open subset $U\subset \bG_k(E,r)$ such that 
$$
X(E,r) \times_{\bG_k(E,r)} U\to  U
$$ 
is smooth projective.
By \cref{Ua_Beilinson_assumption}, the closed immersion
$$
X(E,r) \times_{\bG_k(E,r)} U \hookrightarrow \Fl_k(E,1,r) \times_{\bG_k(E,r)} U
$$
satisfies \cref{relative_beilinson_assumption} above $U$.
Put $X(E,r)_U :=X(E,r)\times_{\bG_k(E,r)} U$ endowed with the stratification induced by $\Sigma$.
By \cref{Bounddmdeg_normal} applied to  $X(E,r)_U \to U$, there is an admissible function $\fc : \bQ[\Coh(X(E,r)_U)]\to \bQ$ and $\fd \in \bN[x]$   such that for every finite local ring $\Lambda$ of residue characteristic $\ell \neq p$, every  $F \in U(k)$, every $\cE\in \bQ[\Coh(X)]$ and every $\cK\in D_{\Sigma|_{X \cap F},\ft}^{[a,b]}(X \cap F,\cE|_{X \cap F},\Lambda)$, we have 
$$
|\totdeg_{i_F}(CC(\cK)) | \leq \fd(\fc(\cE|_{X(E,r)_U})) \cdot \Rk_{\Lambda} \cK \ .
$$
The conclusion thus follows.
\end{proof}

\begin{defin}\label{C_transversality}
Let $X$ be a projective  scheme over $k$.
Let $i : X\hookrightarrow \bP_k(E)$ be a closed immersion, let $\Tr$ be a rooted  tree and let $\cC\subset D_{ctf}^b(X,\Lambda)$ be a full subcategory.
We say that  $F_{\bullet}\in \Fl_k(E,\Tr)(k)$ is \textit{$\cC$-transversal} if for every $\cK \in \cC$, the tree $\Tr$ admits a branch $B$ such that for any consecutive vertices $v\leq w$ in $B$, the map $F_w \hookrightarrow  F_v$ is $SS((i_{\ast}\cK)|_{F_v})$-transversal.
\end{defin}

\begin{rem}
We will abuse the terminology and refer to a branch $B\subset \Tr$ as in \cref{C_transversality} as a \textit{$\cK$-transversal branch of $F_{\bullet}$}.
\end{rem}

\begin{rem}
     Observe that \cref{C_transversality} only asks for $F_w \hookrightarrow  F_v$ to be $SS((i_{\ast}\cK)|_{F_v})$-transversal, as opposed to \textit{proper} $SS((i_{\ast}\cK)|_{F_v})$-transversal.
     In particular, we don't claim any relationship between $SS((i_{\ast}\cK)|_{F_v})$ and $SS(i_{\ast}\cK)$.
\end{rem}

\begin{rem}\label{perv_transversal}
If $\cP\in \Perv_{tf}(X,\Lambda)$ and if $F_{\bullet}\in \Fl_k(E,\Tr)(k)$ admits a  $\cP$-transversal branch $B\subset \Tr$, then $\cP|_{X\cap F_v}[-d(v)]$ is perverse for every $v\in B$ by \cite[Lemma 8.6.5]{cc}.
\end{rem}

For the notion of ramified enough rooted  tree, let us refer to \cref{ramified_tree}.

\begin{thm}\label{micro_local_Lefschetz}
Let $(X,\Sigma)$ be a projective stratified scheme over $k$ algebraically closed.
Then, there is a closed immersion $i : X \hookrightarrow \bP_k(E)$, an admissible function $\fc : \bQ[\Coh(X)]\to \bQ$ and $\fd\in \mathds{N}[x]$  such that for every $2\leq a \leq \rk E$, every $U_a \subset \bG_k(E,a)$ dense open, every $r\geq 1$, every $\cE\in \bQ[\Coh(X)]$ and every rooted  tree $\Tr$ ramified enough with respect to $(\rk E,\fd(\fc(\cE) ) \cdot r)$,  there is a dense open 
subset $U \subset \Fl_k(E,\Tr)$ satisfying the following property : for every finite local ring $\Lambda$ of residue  characteristic $\ell \neq p$ and every $F_{\bullet}\in U(k)$, we have :
\begin{enumerate}\itemsep=0.2cm
\item $F_v\in U_{\rk E-d(v)}$ for every $v\in \Tr$.
\item $F_{\bullet}$ is $\Perv_{\Sigma,tf}^{\leq r}(X,
\cE,\Lambda)$-transversal.
\end{enumerate}
\end{thm}

\begin{proof}
Let $\iota : X\hookrightarrow Y$ be a closed embedding in some smooth projective scheme over $k$.
Let $\Sigma_Y$ be the stratification on $Y$ given by $Y-X$ and the $Z\in \Sigma$.
Choose a closed immersion $i_Y : Y \hookrightarrow \bP_k(E)$ satisfying \cref{relative_beilinson_assumption} and assume that \cref{micro_local_Lefschetz} holds for $(Y,\Sigma_Y,i_Y)$. 
Let 
$\fc_Y : \bQ[\Coh(Y)]\to \bQ$ and $\fd_Y\in \mathds{N}[x]$ as given by \cref{micro_local_Lefschetz} applied to $(Y,\Sigma_Y,i_Y)$.
If we put $i := i_Y \circ \iota :  X \hookrightarrow \bP_k(E)$, then $\fc_Y \circ \iota_* : \bQ[\Coh(X)]\to \bQ$ and $\fd_Y\in \mathds{N}[x]$ do the job.
Hence, we are left to prove \cref{micro_local_Lefschetz} in the case where $X$ is smooth over $k$ and we have a closed immersion $i : X \hookrightarrow \bP_k(E)$ satisfying \cref{relative_beilinson_assumption}.\\ \indent
By \cref{bounded_degree_on_subspace}, there is an admissible function $\fc : \bQ[\Coh(X)]\to \bQ$, a polynomial $\fd\in \mathds{N}[x]$  and dense open subsets $V_a\subset \bG_k(E,a)$ for $a=2,\dots, \rk E$ such that for every finite local ring $\Lambda$ of residue characteristic $\ell \neq p$, every $r\geq 1$ and $\cE\in \bQ[\Coh(X)]$,  every $F \in V_a(k)$, the scheme $X \cap F$  is smooth  and for every $\cP\in \Perv_{\Sigma|_{X \cap F},tf}^{\leq r}(X \cap F,\cE|_{X \cap F},\Lambda)$, we have
$$
|\totdeg_{i_{F}}(CC(\cP)) | \leq \fd(\fc(\cE)) \cdot r 
$$
where $i_F : X\cap F \hookrightarrow F$ is the inclusion.
Fix $r\geq 1$ and $\cE\in \bQ[\Coh(X)]$ and  $U_a \subset \bG_k(E,a)$ dense open for $2\leq a \leq \rk E$.
Let $\Tr$ be a rooted  tree ramifying enough with respect to $(\rk E, \fd(\fc(\cE))\cdot r)$. 
Consider the morphism
$$
\Fl_k(E,\Tr) \to \prod_{v\in \Tr\setminus \{0\}} \bG_k(E,\rk E - d(v))
$$
induced by evaluation at $v\in \Tr \setminus \{0\}$ and let  $A\subset \Fl_k(E,\Tr)$ be the inverse image of 
$$
\prod_{v\in \Tr\setminus \{0\}} U_{\rk E - d(v)}\cap V_{\rk E - d(v)} \ .
$$
By \cref{dense_invers_image}, the subset $A$ is dense open in $\Fl_k(E,\Tr)$.
By \cref{multiple_avoidance_hypersurface}, there is a dense open subset $B \subset \Fl_k(E,\Tr)$ such that for every  $F_{\bullet}\in B(k)$ and every non maximal vertex $v\in \Tr$, no hypersurface of $F_v^{\vee}$ of degree at most $\fd(\fc(\cE)) \cdot r$  contains all the $F_w$,  for $w$ immediate successor of $v$.
Put $U:=A\cap B$. 
It is enough to show that $U$ satisfies (2).\\ \indent
Let $\Lambda$  be a finite local ring of residue  characteristic $\ell \neq p$, let $F_{\bullet}\in U(k)$ and let $\cP\in \Perv_{\Sigma,tf}^{\leq r}(X,\cE,\Lambda)$.
We construct the branch $B\subset \Tr$ step by step by starting from the initial object 0.
Suppose that we have found $v\in \Tr$ non maximal such that for any consecutive vertices $v'\leq w'$ smaller than $v$, the map  
$F_{w'} \hookrightarrow  F_{v'}$ is $SS((i_{\ast}\cP)|_{F_{v'}})$-transversal.
By \cref{perv_transversal}, the complex $(i_{\ast}\cP)|_{F_v}[-d(v)]$ is perverse.
Hence, 
$$
\cP|_{X\cap  F_v}[-d(v)] \in 
\Perv_{\Sigma|
_{X \cap  F_v},tf}^{\leq r}(X\cap F_v,\cE|
_{X\cap  F_v},\Lambda) \ .
$$
Let     
$$
	\begin{tikzcd}
F_v^{\vee}& Q_v \arrow[l,"p_v^{\vee}"]  \arrow{r}{p_v}&   F_v   
	\end{tikzcd}
$$
be the universal hypersurface of $F_v$.
Since $F_{\bullet}$ lies in $A$, \cref{bound_deg_SS} gives 
$$
\deg p_v^{\vee}(\bP(SS((i_{\ast}\cP)|_{F_v}))) \leq |\totdeg_{i_{F_v}}(CC(\cP|_{X\cap F_v}[-d(v)])) |\leq 
\fd(\fc(\cE)) \cdot r \ .
$$
Since $F_{\bullet}$ lies in $B$,  the non maximal vertex $v$ admits an immediate successor $w$ such that 
$$
F_w\notin p_v^{\vee}(\bP(SS((i_{\ast}\cP)|_{F_v})))\subset F_v^{\vee} \ .
$$
By \cref{transversality_H}, the inclusion $F_w \hookrightarrow F_v$ is thus $SS((i_{\ast}\cP)|_{F_v})$-transversal.
 \end{proof}

The condition (1) from  \cref{micro_local_Lefschetz} provides some flexibility in the choice of the subspaces $F_v$.
The flexibility needed in the present paper suggests to introduce the following

\begin{defin}\label{good_position}
Let $U$ be a smooth scheme of finite type over $k$ of pure dimension $n\geq 2$ and let $i : U \hookrightarrow \bP_k(E)$ be an immersion.
Let $\Tr$ be a rooted  tree of depth $\leq n-1$.
We say that $F_{\bullet} \in  \Fl_k(E,\Tr)(k)$ is in \textit{good position with respect to $U$} if for every  $v\in \Tr$, for every irreducible component $Z$ of $U$, the space $F_v$ is transverse to $Z$ and  $Z\cap F_v$ is irreducible.
\end{defin}

\begin{rem}
In the setting of \cref{good_position}, assume that $k$ is algebraically closed.
By Bertini's theorem \cite[Théorème 6.3]{Jouanolou}  and \cref{dense_invers_image}, there is a dense open  subset $V\subset  \Fl_k(E,\Tr)$ such that every  $F_{\bullet} \in  V(k)$ is in good position with respect to $U$.
\end{rem}


%
%

\subsection{Lefschetz recognition principle}



\begin{lem}\label{transversality_and_Hom_branch}
Let $X$ be a projective scheme of finite type of pure dimension $n\geq 2$ over $k$ algebraically closed.
Let $j : U\hookrightarrow X$ be an affine open immersion with $U$ smooth over $k$.
Let $i : X\hookrightarrow \bP_k(E)$ be a closed immersion and let $\cL,\cM\in \Loc_{\ft}(U,\Lambda)$.
Let $B$ be a single-branched rooted  tree of depth at most $n-1$ and let $F_{\bullet}\in Fl_k(E,B)(k)$  in good position with respect to $U$ for which $B$ is $j_!\cHom(\cL,\cM)$-transversal.
Then,  the canonical morphism
$$
\Hom(\cL,\cM)\to \Hom(\cL|_{U\cap F_v},\cM|_{U\cap F_v})
$$
is an isomorphism for every $v\in B$.
\end{lem}

\begin{proof}
Let $v,w\in B$ such that $w$ is an immediate successor of $v$.
It is enough to show that the canonical morphism
$$
\Hom(\cL|_{U\cap F_v},\cM|_{U\cap F_v})\to \Hom(\cL|_{U\cap F_w},\cM|_{U\cap F_w})
$$
is an isomorphism.
Consider the following commutative diagram with cartesian squares
$$
	\begin{tikzcd}
	        U\cap F_v    \arrow{d}   \arrow{r}{j_v} &    X\cap F_v    \arrow{d}   \arrow{r}{i_v}& F_v   \arrow{d} \\
   U\arrow{r}{j}      &    X  \arrow{r}{i}  &  \bP_k(E)\ .
	\end{tikzcd}
$$
We have 
$$
(i_*j_!(\cHom(\cL,\cM)))|_{F_v} \simeq i_{v*}j_{v!}(\cHom(\cL|_{U\cap F_v},\cM|_{U\cap F_v})) \ .
$$ 

On the other hand, $U$ is smooth, $F_{\bullet}$ is in good position with respect to $U$ and the inclusion $F_w\hookrightarrow F_v$ is $SS((i_*j_!(\cHom(\cL,\cM)))|_{F_v})$-transversal.
The conclusion then follows from \cref{transversality_and_Hom}.
\end{proof}

\begin{defin}\label{realizes_Lefschetz_recognition}
Let $X$ be a scheme of finite type over $k$.
Let $i : X\hookrightarrow \bP_k(E)$ be an immersion, let $\Tr$ be a rooted tree and let $\cC\subset D^b_{ctf}(X,\Lambda)$ be a full subcategory.
We say that  $F_{\bullet}\in \Fl_k(E,\Tr)(k)$ \textit{realizes the Lefschetz recognition principle for $\cC$} if for every  $\cK_1,\cK_2 \in \cC$, there is a branch $B\subset \Tr$ such that for 
$a,b\in \{1,2\}$, the canonical morphism
$$
\Hom(\cK_a,\cK_b)\to \Hom(\cK_a|_{U\cap F_v},\cK_b|_{U\cap F_v})
$$
is an isomorphism for every $v\in B$.
\end{defin}

\begin{rem}\label{branch_distinguish}
If $B$ is a branch as in \cref{realizes_Lefschetz_recognition}, then for every $v\in B$,  we have $\cK_1\simeq\cK_2$   if and only if  $\cK_1|_{U\cap F_v} \simeq \cK_2|_{U\cap F_v} $.
\end{rem}

\begin{thm}\label{Lefschetz_recognition}
Let $X$ be a projective scheme of pure dimension $n\geq 2$ over $k$ algebraically closed.
Let $j : U\hookrightarrow X$ be an affine open immersion with $U$ smooth over $k$.
Then, there is a closed immersion $i : X\hookrightarrow \bP_k(E)$, there is an admissible function $\fc : \bQ[\Coh(X)]\to \bQ$ and $\fd\in \mathds{N}[x]$ depending only on $(X,U,i)$ such that for every $r\geq 1$, every $\cE\in \bQ[\Coh(X)]$ and every rooted  tree $\Tr$ of depth $\leq n-1$ ramified enough with respect to $(\rk E,\fd(\fc(\cE) ) \cdot 4r^2)$, there is a dense open subset $V \subset \Fl_k(E,\Tr)$ satisfying the following property : for every finite local ring  $\Lambda$ of residue characteristic $\ell \neq p$,  every $F_{\bullet} \in V(k)$ is in good position with respect to $U$ and realizes the Lefschetz recognition principle for  $\Loc^{\leq r}_{tf}(U,\cE,\Lambda)$.
\end{thm}

\begin{proof}
Put $\Sigma:=\{U,D\}$, let $i : X\hookrightarrow \bP_k(E)$ and $\fc : \bQ[\Coh(X)]\to \bQ$ and $\fd\in \mathds{N}[x]$ as given by  \cref{micro_local_Lefschetz} applied to $(X,\Sigma)$.
Let  $r\geq 1$ and $\cE\in \bQ[\Coh(X)]$.
Let $\Tr$ be a rooted  tree of depth smaller than $n-1$ ramifying enough with respect to $(\rk E,\fd(\fc(\cE) ) \cdot 4r^2)$.
Then, there is a dense open subset $V\subset \Fl_k(E,\Tr)$ such that for  every $F_{\bullet} \in V(k)$ and every finite local ring  $\Lambda$ of residue characteristic $\ell \neq p$, we have
\begin{enumerate}\itemsep=0.2cm
\item $F_{\bullet}$  is in good position with respect to $U$.

\item $F_{\bullet}$ is $\Perv_{\Sigma,tf}^{\leq 4r^2}(X,
\cE,\Lambda)$-transversal.
\end{enumerate}
Let $F_{\bullet} \in V(k)$ and $\cL_1,\cL_2\in \Loc^{\leq r}_{tf}(U,\cE,\Lambda)$.
By \cref{transversality_and_Hom_branch}, it is enough to find a branch $B$  which is $j_!\cHom(\cL_{a},\cL_{b})$-transversal for every $a,b\in \{1,2\}$.
Since 
$$
\bigoplus_{a,b\in \{1,2\}} j_!\cHom(\cL_{a},\cL_{b})[n]  
$$
is an object of $\Perv_{\Sigma,ft}^{\leq 4r^2 }(X,\cE,\Lambda)$ by  \cref{ramification _of_Hom} and the fact that $j : U\hookrightarrow X$ is affine,  the existence of $B$ follows from the property (2).
\end{proof}

\begin{rem}
By exploiting the fact that the dense open subset $V$ does not depend on $\Lambda$,  \cref{Lefschetz_recognition} upgrades easily to $\Qlbar$ coefficients.
\end{rem}

\begin{rem}
An analogue of \cref{Lefschetz_recognition} for flat bundles in characteristic 0 was obtained in \cite{HuTeyssierDRBoundedness}.
\end{rem}

\subsection{Wild Lefschetz in higher codimension}

\begin{defin}\label{realizes_wild_Lefschetz}
Let $U$ be a connected scheme of finite type of pure dimension $n\geq 2$ over  $k$.
Let $i : U\hookrightarrow \bP_k(E)$ be an immersion, let $\Tr$ be a rooted  tree of depth  $\leq n-1$ and let $\cC\subset \Loc_{tf}(U,\Lambda)$ be a full subcategory.
We say that  $F_{\bullet}\in \Fl_k(E,\Tr)(k)$ \textit{realizes the wild Lefschetz theorem for $\cC$} if  $F_{\bullet}$ is in good position with respect to $U$ and for every $\cL\in \cC$, there is a branch $B\subset \Tr$ such that for every $v\in B$, the sheaves $\cL$ and $\cL|_{U\cap F_{v}}$ have the same monodromy group.
We say that such $B$ \textit{preserves the monodromy group of $\cL$}.
\end{defin}

\begin{lem}\label{transversality_and_monodromy}
Let $X$ be a projective scheme of pure dimension $n\geq 2$ over $k$ algebraically closed.
Let $j : U\hookrightarrow X$ be an affine open immersion with $U$ smooth connected over $k$ and let $\ubar\to U$ be a geometric point.
Let $i : X\hookrightarrow \bP_k(E)$  be a closed immersion, let $\Tr$ be a rooted  tree of depth smaller than $n-1$ and let $F_{\bullet}\in \Fl_k(E,\Tr)(k)$ in good position with respect to $U$.
Let $\cL\in \Loc_{tf}(U,\Lambda)$.
Assume that $F_{\bullet}$ admits a $j_!\Lambda[\cL,\ubar]$-transversal branch $B$.
Then, $B$ preserves the monodromy group of $\cL$.
\end{lem}
\begin{proof}
Let $v\in B$.
We have to show that $\cL$ and $\cL|_{U\cap F_{v}}$ have the same monodromy group.
By \cref{sheaf_criterion_same_monodromy}, we have to show that the canonical morphism
$$
H^0(U,\Lambda[\cL,\ubar]) \to H^0(U\cap F_v,\Lambda[\cL,\ubar]|_{U\cap F_v}) 
$$
is an isomorphism.
By \cref{transversality_and_Hom_branch}, the conclusion follows from the fact that $B$ is  $j_!\Lambda[\cL,\ubar]$-transversal.
\end{proof}

\begin{thm}\label{Wild_Lefschetz}
Let $X$ be a projective  scheme of pure dimension $n\geq 2$ over $k$ algebraically closed.
Let $j : U\hookrightarrow X$ be an affine open immersion with $U$ smooth connected over $k$.
Then, there is a closed immersion $i : X\hookrightarrow \bP_k(E)$, there is an admissible  function $\fc : \bQ[\Coh(X)]\to \bQ$ and $\fd\in \mathds{N}[x]$ such that for every $r\geq 1$, every $\cE\in \bQ[\Coh(X)]$ and every rooted  tree $\Tr$ of depth $\leq n-1$ ramifying enough with respect to $(\rk E,\fd(\fc(\cE))\cdot |\GL_r(\Lambda)|)$, there is a dense open subset 
$V\subset \Fl_k(E,\Tr)$ with the property that every $F_{\bullet} \in V(k)$ realizes the wild Lefschetz theorem for $\Loc^{\leq r}_{tf}(U,\cE,\Lambda)$.
\end{thm}

\begin{proof}
Let $i : X\hookrightarrow \bP_k(E)$, let $\fc : \bQ[\Coh(X)]\to \bQ$ and $\fd\in \mathds{N}[x]$ as given by  \cref{micro_local_Lefschetz} applied to $(X,\Sigma)$ where $\Sigma:=\{U,X-U\}$.
Let  $r\geq 1$ and $\cE\in \bQ[\Coh(X)]$.
Let $\Tr$ be a rooted  tree of depth  $\leq n-1$  ramifying enough with respect to $(\rk E,\fd(\fc(\cE))\cdot |\GL_r(\Lambda)|)$ and let $V\subset \Fl_k(E,\Tr)$ be a dense open subset such that for every $F_{\bullet} \in  V(k)$, we have 
\begin{enumerate}\itemsep=0.2cm
\item $F_{\bullet}$ is in good position with respect to  $U$.

\item $F_{\bullet}$ is $\Perv_{\Sigma,tf}^{\leq |\GL_r(\Lambda)|}(X,
\cE,\Lambda)$-transversal.
\end{enumerate}
Let $F_{\bullet} \in V(k)$ and let $\cL\in \Loc^{\leq r}_{tf}(U,\cE,\Lambda)$.
We have to show the existence of a branch preserving the monodromy group of $\cL$.
Given a geometric point $\ubar\to U$, we have to show by \cref{transversality_and_monodromy} that $F_{\bullet}$ admits a $j_!\Lambda[\cL,\ubar]$-transversal branch.
Since $j_!\cL$ has log conductors bounded by $\cE$, \cref{representation_local_system_lc_bound} implies that $j_!\Lambda[\cL,\ubar]$ has log conductors bounded by $\cE$ as well.
Since $j : U\hookrightarrow X$ is affine, we deduce
$$
j_!\Lambda[\cL,\ubar][n]\in \Perv_{\Sigma, tf}(X,\cE,\Lambda) \ .
$$
Since $j_!\Lambda[\cL,\ubar][n]$  has rank smaller than $|\GL_r(\Lambda)|$, the existence of the sought-after branch follows by construction of $V$.

\end{proof}

Similarly as in \cref{thm_5bis}, if one wishes to ignore the quantitative aspect of \cref{thm_5},  one gets the following
\begin{cor}
In the setting of \cref{Wild_Lefschetz},  there is a closed  immersion $i : X\hookrightarrow \bP_k$  such that for every prime $\ell \neq p$,  every $r\geq 0$, every $1\leq d\leq \dim \bP_k -1$ and every effective Cartier divisor $R$ supported on $D$,  there exist linear subspaces $(F_1,\dots, F_{N})$ of dimension $d$ with $N$ depending on $\ell,  r$ and $R$ such that for every  $\cL\in  \Loc(U,\bF_{\ell})$ of rank $\leq r$ with log conductors bounded by $R$,  one of the $F_i$ preserves the monodromy group of $\cL$.
\end{cor}


\begin{thebibliography}{{Sai}17b}




\bibitem[AS02]{RamImperfect}
A.~{Abbes} and T.~{Saito}, \emph{{{Ramification of local fields with imperfect
  residue fields}}}, {{Amer. J. Math.}} \textbf{{124}} ({2002}).

\bibitem[AS03]{as ii}
\bysame, 
\emph{Ramification of local fields with imperfect residue fields II}.
Doc. Math. Extra Volume Kato, (2003).


\bibitem[AS11]{Ram_and_clean}
\bysame, \emph{Ramification and cleanness}, {Tohoku Math. J.} \textbf{{63}} ({2011}).


\bibitem[SGA4III]{SGA4-3}
M.~{Artin}, A.~{Grothendieck}, and J.-L {Verdier}, \emph{Th{\'e}orie des
  {T}opos et {C}ohomologie {E}tale des {S}ch{\'e}mas}, Lecture {N}otes in
  {M}athematics, vol. 305, Springer-{V}erlag, 1973.


\bibitem[{Bei}16]{bei}
A.~{Beilinson}, \emph{{Constructible sheaves are holonomic}}, {Sel. Math. New.
  Ser.} \textbf{{22}} ({2016}).


\bibitem[BBDG18]{BBD}
A.~{Beilinson}, J.~{Bernstein}, and P. {Deligne} and O. {Gabber}, \emph{Faisceaux pervers}, Astérisque, vol. 100, {SMF}, 2018.










  
 \bibitem[Del87]{FinMon}
P. ~{Deligne},  \emph{Un théorème de finitude pour la monodromie}, in Discrete Groups in Geometry and Analysis (New Haven, 1984), Progr. Math. {\bf 67}, Birkhäuser Boston, (1987). 

 



\bibitem[{Del}16]{DeltoEsn}
\bysame, \emph{{Letter to E. Esnault}}, {February 29th, 2016}.
  
  
  
\bibitem[D12]{DriDel}
V. {Drinfeld}, \emph{
On a conjecture of Deligne}, {Moscow Mathematical Journal} \textbf{{12}} ({2012}).





\bibitem[E17]{ELefschetz}
H. {Esnault}, \emph{
Survey on some aspects of Lefschetz theorems in algebraic geometry
}, {Revista matemática complutense} \textbf{{30}} ({2017}).




\bibitem[{EK}16]{EKin}
H. {Esnault} and L. {Kindler}, \emph{Lefschetz theorems for tamely ramified coverings}, {Proceedings of the American Mathematical Society} \textbf{{144}} ({2016}).



\bibitem[{ES}21]{ES}
H. {Esnault} and V. {Srinivas}, \emph{Bounding ramification by covers and curves}, {Proceedings of Symposia in Pure Mathematics} \textbf{{104}} ({2021}).

\bibitem[{Fu}11]{fu}
L.~{Fu}, \emph{{{Etale Cohomology Theory}}}, {{Nankai Tracts in Mathematics}},
  vol.~{13}, {2011}.


\bibitem[{GM}88]{GM}
M. {Goresky} and R. {MacPherson}, \emph{Stratified Morse Theory}, {Ergebnisse der Mathematik und ihrer Grenzgebiete}, {Springer,} {\bf 14}, (1988).



\bibitem[{Gro}68]{SGA2}
A. {Grothendieck}, \emph{Cohomologie locale des faisceaux cohérents et théorème de Lefschetz
locaux et globaux (SGA2)}, {Advanced Studies in Pure Math., North-Holland Amsterdam, vol. 2} ({1968}).



\bibitem[GD61]{EGA3-1}
\bysame,  \emph{El{\'e}ments de {G}{\'e}om{\'e}trie {A}lg{\'e}brique {III}}, Publications {M}ath{\'e}matiques de l'{IHES}, vol. 11,  1961.




\bibitem[GD64]{EGAIV}
\bysame,  \emph{El{\'e}ments de {G}{\'e}om{\'e}trie {A}lg{\'e}brique {IV}}, Publications {M}ath{\'e}matiques de l'{IHES}, vol. 20,24,28,32,  1966.


\bibitem[{HL}85]{HL}
H. {Hamm} and L. {Dũng Tráng}, \emph{Lefschetz theorems on quasi-projective varieties}, {Bulletin de la Société Mathématique de France} \textbf{{113}} ({1985}).








\bibitem[Hir17]{Hiranouchi}
T. {Hiranouchi}, \emph{A Hermite–Minkowski type theorem of varieties over finite fields}, {Journal of Number Theory} \textbf{{176}} ({2017}).

\bibitem[HT21]{HT}
H. {Hu}, J.-B. {Teyssier}, \emph{Characteristic cycle and wild ramification for nearby cycles of étale sheaves}, {J. für die Reine und Angew. Math.} \textbf{{776}} ({2021}).

\bibitem[HT24]{HuTeyssierDRBoundedness}
\bysame, \emph{Cohomological boundedness for flat bundles on surfaces and applications},
{Compositio Mathe\-matica.} \textbf{{160}}  ({2024}).


\bibitem[HT25a]{HuTeyssierSemicontinuity}
\bysame, \emph{Semi-continuity for conductors of étale sheaves},
{International Mathematics Research Notices.}  \textbf{{2026}} ({2026}).

\bibitem[HT25b]{HuTeyssierCohBoundedness}
\bysame, \emph{Bounding ramification with coherent sheaves},
{Tunisian J. Math.} \textbf{{8}} ({2026}).


\bibitem[HT25c]{HuTeyssierMinkowski}
\bysame, \emph{Estimates for Betti numbers and relative Hermite-Minkowski theorem for perverse sheaves},
{Preprint}, 2025.






\bibitem[Jou83]{Jouanolou}
J.-P.~{Jouanolou}, \emph{{Théorèmes de Bertini et applications}}, {Progress in Math., Birkhäuser Boston} \textbf{{42}} ({1983}).


\bibitem[Ka88]{KatzGauss}
N.~{Katz}, \emph{{Gauss sum, Kloosterman sums, and monodromy groups}}, {Ann. of Math. Stud., Princeton
University Press} \textbf{{116}} ({1988}).






\bibitem[KS90]{Kashiwara_Schapira}
M.~{Kashiwara} and P. {Schapira}, \emph{Sheaves on manifolds,} Springer-{V}erlag, (1990).




\bibitem[KS14]{KS}
M.~{Kerz} and S. {Saito}, \emph{Lefschetz theorem for abelian fundamental group with modulus,} Algebra Number Theory, {\bf 8}, (2014).













\bibitem[{Sai}08]{logcc}
T.~{Saito}, \emph{Wild ramification and the characteristic cycles of an $\ell$-adic sheaf,} J. of Inst. of Math. Jussieu, {\bf 8} 4, (2008).


\bibitem[{Sai}17a]{wr}
\bysame, \emph{{Wild ramification and the cotangent bundle}}, {{J. Algebraic
  Geom.}} \textbf{{26}} ({2017}).

\bibitem[{Sai}17b]{cc}
\bysame, \emph{{The characteristic cycle and the singular support of a
  constructible sheaf}}, {{Invent. Math.}} \textbf{{207(2)}} ({2017}).

\bibitem[{Sai}21]{saito21}
\bysame, \emph{Characteristic cycles and the conductor of direct image,} Journal of the American Mathematical Society \textbf{34}, (2021).


\bibitem[{Sai}22]{saito22}
\bysame, \emph{Cotangent bundle and microsupports in the mixed characteristic case,} Algebra and Number Theory, \textbf{16}:2, (2022).

\bibitem[{SY}17]{SY}
T.~{Saito} and Y. {Yatagawa}, \emph{Wild ramification determines the characteristic cycle}, {Ann. Sci. Éc. Norm. Supér.} \textbf{{50}} ({2017}).
  
  \bibitem[{Ser}68]{CL}
J.-P. {Serre}, \emph{Corps {L}ocaux}, Herman, 1968.

\bibitem[SP23]{SP}
The Stacks project authors, The Stacks project, \url{https://stacks.math.columbia.edu}, 2023.


\bibitem[{UYZ}20]{UYZ}
N.~{Umezaki} and E. {Yang} and Y. {Zhao}, \emph{Characteristic class and the $\epsilon$-factor of an étale sheaf,}  {Trans. Amer. Math. Soc.} \textbf{{373}} ({2020}).
  

\bibitem[Wei94]{Wei}
C.~{Weibel}, \emph{An Introduction to Homological Algebra,}  {Cambridge University Press} \textbf{{38}} ({1994}).
    
    
    
\end{thebibliography}
\end{document}